\documentclass[preprint,11pt,authoryear]{elsarticle}

\usepackage[utf8]{inputenc}
\usepackage{geometry}
\usepackage[utf8]{inputenc}
\usepackage{graphicx} 
\usepackage{bussproofs}
\usepackage{amsmath,amssymb,amsfonts}
\usepackage{mathrsfs,latexsym,amsthm,enumerate}
\usepackage{tikz}
\usepackage{tikz-cd}
\usetikzlibrary{cd}
\usepackage{natbib}
\usepackage{amscd}
\usepackage[all]{xy}
\usepackage[utf8]{inputenc}
\usepackage{amsmath,amssymb,amsfonts}
\usepackage{mathrsfs,latexsym,amsthm,enumerate}
\usepackage{breqn}
\usepackage{fancyhdr}
\usepackage{sectsty}
\usepackage{stmaryrd}
\usepackage{url}
\usepackage{bbm}
\usepackage{subcaption}
\usepackage{float}
\catcode`@=11
\def\caseswithdelim#1#2{\left#1\,\vcenter{\normalbaselines\m@th
  \ialign{\strut$##\hfil$&\quad##\hfil\crcr#2\crcr}}\right.}
\catcode`@=12

\allsectionsfont{\sffamily\mdseries\upshape}

\allsectionsfont{\sffamily\mdseries\upshape}

\theoremstyle{plain}
\newtheorem{lemma}{Lemma}[section]
\newtheorem{proposition}[lemma]{Proposition}
\newtheorem{corollary}[lemma]{Corollary}

\newtheorem{theorem}[lemma]{Theorem}

\theoremstyle{definition}
\newtheorem{remark}[lemma]{Remark}

\newcommand{\lra}{\leftrightarrows}
\newcommand{\ra}{\rightarrow}

\newcommand{\inclu}{\hookrightarrow}

\newcommand{\up}{{\uparrow}}

\newcommand{\tc}{\textit}
\newcommand{\mb}[1]{\mbox{#1}}
\newcommand{\ca}{\mathcal}
\newcommand{\mf}{\mathsf}
\newcommand{\mc}[1]{\mathit{#1}}
\newcommand{\mi}{\mathit}

\newcommand{\se}{\subseteq}
\newcommand{\sm}{\setminus}

\newcommand{\we}{\wedge}
\newcommand{\ve}{\vee}

\newcommand{\bve}{\bigvee}
\newcommand{\cu}{\cup}
\newcommand{\bca}{\bigcap}
\newcommand{\bcu}{\bigcup}
\newcommand{\bd}[1]{\mathbf{#1}}

\newcommand{\op}{\mathfrak{o}}
\newcommand{\cl}{\mathfrak{c}}
\newcommand{\Fi}{\mathsf{Filt}}
\newcommand{\Fip}{\mathsf{Filt}(L^+)}
\newcommand{\Fim}{\mathsf{Filt}(L^-)}

\newcommand{\less}[2]{\nabla (#1)\cap\Delta (#2)}
\newcommand{\na}{\nabla}
\newcommand{\del}{\Delta}

\newcommand{\va}[3]{\varphi_{\ca{#1}}^{#2}(#3^{#2})}

\newcommand{\pa}[1] {(#1^+,#1^-)}
\newcommand{\p}[2]{(#1^+,#2^-)}

\newcommand{\Om}{\Omega}

\newcommand{\bpt}{\mf{bpt}}

\newcommand{\bOm}{\mf{b}\Om}

\newcommand{\pt}{\mathsf{pt}}
\newcommand{\Lq}[1]{\ca{L}/#1}

\newcommand{\Acf}{\mathsf{A}_{cf}(\ca{L})}
\newcommand{\Apm}{\mathsf{A}_{\pm}(\ca{L})}

\newcommand{\Ad}{\mf{A}(\ca{L})}

\newcommand{\fin}{\mf{fin}}

\newcommand{\sy}[1]{{{\ulcorner}} #1 {{\urcorner}}}
\newcommand{\syp}[1]{{\ulcorner} #1^+{\urcorner}}
\newcommand{\sym}[1]{{\ulcorner} #1^-{\urcorner}}
\newcommand{\syw}[2]{\syp{#1}\we \sym{#2}}
\newcommand{\syv}[2]{\syp{#1}\ve \sym{#2}}

\newcommand{\cpu}[1]{#1^+\oplus #1^-}

\title{The assembly of a pointfree bispace and its two variations}
\author{Anna Laura Suarez\fnref{address}}
\address{CMUC, Department of Mathematics, University of Coimbra, 3001-454, Coimbra, Portugal}
\address{School of Computer Science, University of Birmingham, B15 2TT, Birmingham, UK}
\ead{axs1431@cs.bham.ac.uk}

\begin{document}

\begin{frontmatter}

\begin{abstract}
The duality of finitary biframes as pointfree bitopological spaces is explored. In particular, for a finitary biframe $\ca{L}$ the ordered collection of all its pointfree bisubspaces (i.e. its biquotients) is studied. It is shown that this collection is bitopological in three meaningful ways. In particular it is shown that, apart from the assembly of a finitary biframe $\ca{L}$, there are two other structures $\Acf$ and $\Apm$, which both have the same main component as $\Ad$. The main component of both $\Acf$ and $\Apm$ is the ordered collection of all biquotients of $\ca{L}$. The structure $\Acf$ being a biframe shows that the collection of all biquotients is generated by the frame of the \tc{patch-closed} biquotients together with that of the \tc{patch-fitted} ones. The structure $\Apm$ being a biframe shows the collection of all biquotients is generated by the frame of the \tc{positive} biquotients together with that of the \tc{negative} ones.
Notions of fitness and subfitness for finitary biframes are introduced, and it is shown that the analogues of both characterization theorems for these axioms appearing in \cite{picadopultr2011frames} hold. A spatial, bitopological version of these theorems is proven, in which finitary biframes whose spectrum is pairwise $T_1$ are characterized, among other things in terms of the spectrum $\bpt(\Acf)$.
\end{abstract}


\end{frontmatter}

\tableofcontents 

\section*{Introduction}

Bitopological spaces were introduced in \cite{kelly63} for the study of quasi-metric spaces. In \cite{jung06} and in \cite{bezhanishvili10} it is explained how Stone-type dualities are bitopological in nature, and that therefore it is advantageous to think of these dualities in terms of bitopological spaces. The study of pointfree bispaces, too, has been pursued. In \cite{banaschewski83} structures called \emph{biframes} are introduced as pointfree duals of bitopological spaces. In \cite{jung06} and in \cite{jakl2018}, instead, the category of \emph{d-frames} is studied as a suitable algebraic category representing that of bitopological spaces.

\subsection*{Finitary biframes}
In \cite{suarez2020category}, finitary biframes are introduced as pointfree duals of bitopological spaces. The advantages of adopting these as our notion of pointfree bispace mainly revolve around the theory of sublocales. Biframe quotients simply amount to quotients of the main component of the biframe, and this means that the pointfree bisubspaces of a biframe amount to the pointfree subspaces of its main component. Thus, distinct biframes with the same main component have the same lattice of quotients. Furthermore, the ordered collection $\mf{S}(L)$ of frame quotients of $L$ cannot be equipped with a meaningful bitopological structure inherited from $\ca{L}$.

For finitary biframes, instead, the situation does not collapse to the monotopological case. A \tc{biquotient} of a finitary biframe is simply a biframe quotient which is finitary. Finitary biframes with the same main component may have different lattices of biquotients. For every finitary biframe $\ca{L}$ we have a finitary biframe $\ca{L}$ such that its main component is anti-isomorphic to the coframe of biquotients of $\ca{L}$. The biframe $\mf{A}(\ca{L})$ is equipped with a bitopological structure which is indeed inherited from $\ca{L}$.  

\subsection*{The aim of this article}
For a frame $L$, there exists the frame $\mf{A}(L)$ of the congruences of $L$ ordered under set inclusion. In pointfree topology, this is regarded to be the collection of all pointfree subspaces of $L$ (see \cite{picadopultr2011frames}, \cite{johnstone82}, \cite{isbell72}). Embedded in $\mf{A}(L)$ we have the subframe of closed congruences (representing closed subspaces of the space $L$) and the subframe of \tc{fitted} congruences, namely, the joins of open congruences (those representing the open subspaces of $L$). These two frames generate all of $\mf{A}(L)$. Thus, there is a bitopological interplay between closed and fitted subspaces. 

The central claim of this article is that for a pointfree subspace $\ca{L}$ we still have a bitopological interplay between closed and fitted bisubspaces, but next to this we have a bitopological interplay which is -- so to speak -- orthogonal to this: the interplay between positive and negative bisubspaces of $\ca{L}$. 

For a biframe $\ca{L}=(L^+,L^-,L)$, for brevity we will regard $L$ as a free construction generated by $L^+$ and $L^-$, and thus we will denote the two subframe inclusions $\sy{-}:L^+\ra L$, maps assigning to each generator its corresponding syntactic expression. For a finitary biframe $\ca{L}$, it is known that its biquotients form a coframe, which we call $\ca{S}(\ca{L})$. Let us look, then, at the two natural ways in which the frame $\mf{S}(\ca{L})^{op}$ is bitopological.

\begin{itemize}
    \item All closed congruences of $L$ are contained in $\mf{S}(\ca{L})^{op}$. These form a subframe of this collection, which we call $\mf{Cl}(L)$. All fitted congruences $\bve_i \del(f_i)$, with $f_i\in L$ a \tc{finitary} element, are in $\mf{S}(\ca{L})^{op}$. These, too, form a subframe of $\mf{S}(\ca{L})^{op}$, which we call $\mf{Fitt}_{\fin}(L)$. The subframes of closed congruences and that of finitary fitted ones together generate all of $\mf{S}(\ca{L})^{op}$. This means that the structure
    \[
    (\mf{Cl}(L),\mf{Fitt}_{\fin}(L),\mf{S}(\ca{L})^{op})
    \]
    is a biframe. We call this the \tc{closed-fitted} assembly of the finitary biframe $\ca{L}$.
    \item By \tc{positive} biquotient of a finitary biframe $\ca{L}=(L^+,L^-,L)$ we mean a biframe quotient of the form $\Lq{\syp{C}}$, where $C^+$ is a congruence on $L^+$ and $\syp{C}$ is the congruence $\{(\syp{x},\syp{y}):(x^+,y^+)\in C^+\}$ on $L$. We define \tc{negative} biquotients similarly. Positive biquotients of $\ca{L}$ form a subframe of $\mf{S}(\ca{L})^{op}$, and so do negative biquotients. These generate the whole frame $\mf{S}(\ca{L})^{op}$. This means that for every finitary biframe $\ca{L}$ we have a biframe
\[
(\mf{S}^+(\ca{L})^{op},\mf{S}^-(\ca{L})^{op},\mf{S}(\ca{L})^{op}),
\]
where $\mf{S}^+(\ca{L})$ is the coframe of positive biquotients and $\mf{S}^-(\ca{L})$ that of negative ones. This biframe is called the \tc{positive-negative} assembly of the finitary biframe $\ca{L}$.
\end{itemize}

\section*{Preliminaries}
\subsection*{Biframes}
Biframes are studied in \cite{banaschewski83} and in \cite{schauerte92} as pointfree duals of bitopological spaces. A \emph{biframe} is a triple $(L^+,L^-,L)$ where the three components are all frames, and such that there are frame injections $e^+_L:L^+\ra L$ and $e^-_L:L^-\ra L$ such that every element of $L$ is of the form $\bve_i e^+_L(x^+_i)\we e^-_L(x^-_i)$ for families $x^+_i\in L^+$ and $x^-_i\in L^-$. For a biframe $\ca{L}$, we call $L$ its \tc{main component}. The main component of a biframe represents pointfreely the patch topology of a bispace, while $L^+$ and $L^-$ represent the positive and the negative topologies, respectively. For every biframe $\ca{L}$ the main component is canonically isomorphic to a quotient $(\cpu{L}/C_L)$ of the frame coproduct $\cpu{L}$, in the sense that there is always a congruence $C_L$ on $\cpu{L}$ and an isomorphism $i:L\cong (\cpu{L})/C_L$ such that the following diagram and the analogous one for the negative frame commute. 
\[
\begin{tikzcd}[row sep=large, column sep=large]
L
\arrow{r}{i(\cong)}
& (\cpu{L})/C_L\\
L^+
\arrow{u}{e^+_L}
\arrow{ur}{\sy{-}}
\end{tikzcd}
\]
Here, $\sy{-}$ is the canonical map of generators into the free construction $(\cpu{L})/C_L$. It is in virtue of this fact that, unless otherwise stated, we will regard the main component $L$ of a biframe $\ca{L}$ as a free construction with generators from $L^+\cup L^-$ and we will regard the injections $e^+_L$ and $e^-_L$ as being the maps $\sy{-}:L^+\ra L$ and $\sy{-}:L^-\ra L$ sending each generator to the syntactic expression corresponding to it in the free construction $L$. A map of biframes $f:\ca{L}\ra \ca{M}$ is a pair of maps $f^+:L^+\ra M^+$ and $f^-:L^-\ra M^-$ such that there is a frame map $f:L\ra M$ with $f(\syp{a})=\sy{f^+(a^+)}$ for each $a^+\in L^+$, and similarly for each $a^-\in L^-$. We call $\bd{BiFrm}$ the category of biframes. We have an adjunction $\bOm:\bd{BiTop}\lra \bd{BiFrm}^{op}:\bpt$, in which the two functors act as follows.
\begin{itemize}
    \item For every bispace $(X,\Om^+(X),\Om^-(X))$, the biframe $\mf{b}\Om(X)$ is defined to be the triple 
    \[
    (\Om^+(X),\Om^-(X),\Om(X)),
    \]
    with set inclusions as the two frame injections. Here, $\Om(X)$ is the frame of opens of the topology on $X$ generated by $\Om^+(X)\cup \Om^-(X)$. For a bicontinuous map $f:X\ra Y$, we define the map $\mf{b}\Om(f)$ as the pairing of preimage maps $(f^{-1},f^{-1}):\mf{b}\Om(Y)\ra \mf{b}\Om(X)$. 
    \item For every biframe $\ca{L}$, we define the bispace $\bpt(\ca{L})$ as follows. Its underlying set of points is $\bd{BiFrm}(\ca{L},\bd{2})$, where $\bd{2}$ denotes the biframe $(2,2,2)$. A positive open of $\bpt(\ca{L})$ is any subset of $\bd{BiFrm}(\ca{L},\bd{2})$ of the form $\{\pa{f}\in \bd{BiFrm}(\ca{L}, \bd{2}): f^+(a^+)=1\}$ for some $a^+\in L^+$; negative opens are defined analogously. For every biframe map $f:\ca{L}\ra \ca{M}$ we define the bicontinuous map $\bpt(f):\bpt(\ca{M})\ra \bpt(\ca{L})$ as the pairing of precompositions $(-\circ f^+,-\circ f^-)$.
    \end{itemize}
\subsection*{Quotients of biframes}
We may always quotient a frame $L$ by a relation $R$. This means that for a frame $L$ and for a relation $R$ on it, there is always a quotient $L/R$, whose quotient map we call $q_R:L\ra L/R$, such that the following holds: whenever $f:L\ra M$ is a frame map such that $f(x)\leq f(y)$ whenever $(x,y)\in R$, there is a unique frame map $\tilde{f}:L/R\ra M$ such that $f=\tilde{f}\circ q_R$. We refer as $[x]_R$ to the equivalence class of $x$ in $L/R$, for every $x\in L$. For a biframe $\ca{L}$ and a relation $R$ on its main component $L$, the biframe quotient $\ca{L}/R$ is the biframe $(L^+/R^+,L^-/R^-,L/R)$, where $L^+/R^+$ is the subframe $\{[\syp{x}]_R:x^+\in L^+\}$ of $L$, and $L^-/R^-$ is defined similarly.

\section{Finitary biframes}
If we have a frame $L$ generated by a subset $S\se L$, we say that an element of $L$ is \tc{finitary with respect to }$S$ if it is a finite join of finite meets of generators. ABy distributivity of $L$, every finitary element can also be written as a finite meet of finite joins of generators. For a biframe $\ca{L}$, we call $\fin(L)$ the set of elements of $L$ which are finitary with respect to the set of generators $e^+_L[L^+]\cup e^-_L[L^-]$. Every element in $\fin(L)$, then, can be written as $(\syw{x_1}{x_1})\ve ...\ve (\syw{x_n}{x_n})$ or as $(\syv{y_1}{y_1})\we ...\we (\syv{y_m}{y_m})$. 
When we have a biframe $\ca{L}$, we say that a congruence on $L$ is \tc{finitary} if and only if it is generated by its restriction to $ \fin(L)\times \fin(L)$. We say that a relation on $L$ is \tc{finitary} if it is only constituted of pairs of finitary elements. We say that a biframe $\ca{L}$ is \tc{finitary} if and only if $C_L$ a finitary congruence. 
Recall that a \tc{precongruence} on a frame $L$ (see \cite{moshier17}) is a relation $R\se L\times L$ such that it is reflexive and transitive, and such that it is a subframe of $L\times L$. We recall that congruences are just precongruences which are symmetric relations. For a biframe $\ca{L}$ we say that a precongruence on $L$ is \tc{finitary} if it is generated by its restriction to $\fin(L)\times \fin(L)$.
\begin{lemma}\label{finpre}
A finitary precongruence on $L$ is generated by its restriction to $\fin^{\we}(L)\times \fin^{\ve}(L)$, where $\fin^{\we}(L)=\{\syw{a}{a}:a^+\in L^+,a^-\in L^-\}$ and $\fin^{\ve}=\{\syv{a}{a}:a^+\in L^+,a^-\in L^-\}$.
\end{lemma}
\begin{proof}
Every finitary element of $L$ can be written both as a finite conjunction of binary disjunctions and as a finite disjunction of binary conjunctions. So, whenever a precongruence on $L$ contains the pair $((\syw{x_1}{x_1})\ve ...\ve (\syw{x_n}{x_n}),(\syv{y_1}{y_1})\we ...\we (\syv{y_m}{y_m}))$, it also contains the pairs $\{(\syw{x_a}{x_a},\syv{y_b}{y_b}):a\leq n,b\leq m\}$.
\end{proof}
For a finitary biframe $\ca{L}$, we call $R_L$ the relation $\{(\syw{a}{a},\syv{b}{b}): \syw{a}{a}\leq \syv{b}{b}\mb{ in }L\}$ which generates the congruence $C_L$. For a congruence $C_L$ on $L$, we also call $\mi{fin}(C_L)$ the congruence generated by $C$ restricted to $\fin(L)$. Finitary biframes determine a full subcategory of the category $\bd{BiFrm}$ of biframes, which we call $\bd{BiFrm_{fin}}$.  
In \cite{suarez2020category}, the following is shown.
\begin{lemma}\label{finitarylemma9}
If we have two frames $L^+$ and $L^-$ and a finitary congruence $C$ on $\cpu{L}$, the biframe $(L^+,L^-,\cpu{L})/C=(L^+/C^+,L^-/C^-,(\cpu{L})/C)$ is finitary.
\end{lemma}
The category of finitary biframes is reflective in $\bd{BiFrm}$, and the reflector $\mi{fin}:\bd{BiFrm}\ra \bd{BiFrm_{fin}}$ maps each biframe $\ca{L}$ to the finitary biframe $(L^+,L^-(\cpu{L})/\mi{fin}(C_L))$. By composing the adjunction $\mi{fin}:\bd{BiFrm}\lra \bd{BiFrm_{fin}}:i$ with the adjunction $\bOm:\bd{BiTop}\lra \bd{BiFrm}^{op}:\bot$ for biframes, we obtain an adjunction $\bOm_{\fin}:\bd{BiTop}_{\fin}\lra \bd{BiFrm}^{op}:\bpt$, with $\bpt\dashv \bOm_{\fin}$. This is the adjunction because of which we see finitary biframes as pointfree duals of bitopological spaces.

\section{Bipectra of biquotients}
The results in this section all appear in \cite{suarez2020category}, but we re-state them and sometimes re-prove them here for ease of consultation. For a finitary biframe $\ca{L}$, we say that the biframe quotient $\Lq{R}$ is a \tc{biquotient} if the biframe $\Lq{R}$ is finitary. 
\begin{lemma}\label{witness}
Suppose that $L$ is a frame, and that $R$ is a relation on $L$ and $S$ a relation on $L/R$. Then, suppose that for each equivalence class we have some witness $w(x)\in [x]_R$. The quotient $(L/R)/S$ is the same as the quotient of $L$ by $R\cup \{(w(x),w(y)):([x]_R,[y]_R)\in S\}$.
\end{lemma}
\begin{proof}
We show that the two quotients are the same by showing that a morphism $f:L\ra M$ factors through the first quotient if and only if it factors through the second. Suppose that $M$ is a frame and that $f:L\ra M$ factors through the quotient $(L/R)/S$. This means that whenever $(a,b)\in R$ we have $f(a)\leq f(b)$, and that, additionally, whenever $([x]_R,[y]_R)\in S$ we have $f(x)\leq f(y)$. We claim that $f$ factors through the second quotient. If we have $(x,y)\in R$, we deduce directly from our assumption that $f(x)\leq f(y)$. If $([x]_R,[y]_R)\in S$, we have that $f(w(x))=f(x)$ and $f(w(y))=f(y)$ as $f$ respects the relation $R$. Then, we have $f(w(x))=f(x)\leq f(y)=f(w(y))$, and so $f$ respects the relation $\{(w(x),w(y)):([x]_R,[y]_R)\in S\}$. Conversely, suppose that $f$ factors through the second quotient. Then, $f$ respects $R$. Suppose that $([x]_R,[y]_R)\in S$. Then, by assumption on $f$, we have that $f(x)=f(w(x))\leq f(w(y))=f(y)$, where the two equalities hold because $f$ respects $R$.
\end{proof}
\begin{lemma}\label{finitarylemma91}
Suppose that $\ca{L}$ is a biframe and $C$ is a congruence on $L$, and $R$ a relation on $L$. We have that $L/(C\cup R)$ is isomorphic to the frame $L/C$ quotiented by $[R]_C:=\{([x]_C,[y]_C):(x,y)\in R\}$. Furthermore, if $R$ is finitary, so is $[R]_C$. 
\end{lemma}
\begin{proof}
Let $\ca{L}$ be a biframe, and let $C$ be a congruence on its main component $L$. We show that a frame morphism $f:L\ra M$ factors through the quotient $L/(C\cup R)$ if and only if it factors through the quotient $(L/C)/[R]_C$. If $f:L\ra M$ factors through the first quotient, it determines a map $\tilde{f}:L/C\ra M$ defined as $\tilde{f}([x]_C)=f(x)$ for every $x\in L$. If $([x]_C,[y]_C)$, then form some $x'\in [x]_C$ and some $y'\in [y]_C$ we have $(x',y')$. Then, by assumption on $f$ we have $f(x')\leq f(y')$ and so $\tilde{f}([x]_C)\leq \tilde{f}([y]_C)$. Then, $\tilde{f}$ respects $[R]_C$ and so $f$ factors through $(L/C)/[R]_C$. For the converse, suppose that we have a map $f:L\ra M$ such that $f(x)=f(y)$ whenever $(x,y)\in C$ and such that the canonical map $\tilde{f}:L/C\ra R$ respects the relation $[R]_C$. If we have $(x,y)\in R$, then we also have $\tilde{f}([x]_C)\leq \tilde{f}([y]_C)$, and since $\tilde{f}\circ[-]_C=f$ this implies that $f(x)\leq f(y)$. For the second part of the claim, we notice that the quotient map $[-]_C:L\ra L/C$ respects the frame operations and so it preserves for elements the property of being finitary.
\end{proof}
\begin{proposition}\label{finitaryquotient}
For a finitary biframe $\ca{L}$, we have that the biframe quotient $\Lq{R}$ is a biquotient if and only if $R$ induces a finitary congruence on $L$.
\end{proposition}
\begin{proof}
Suppose that $R$ induces a finitary congruence on $L$. We need to show that the biframe $(L^+/R^+,L^-/R^-,L/R)$ is finitary, and this holds by Lemma \ref{finitarylemma9}. For the converse, suppose that $\ca{L}/R$ is finitary. Then, its third component is the quotient of the frame $(\cpu{L})/(\syp{R}\cup \sym{R})$ by a congruence induced by some finitary relation $R':=\{([\syp{x_j}]_R\we[\sym{x_j}]_R,[\syp{y_j}]_R\ve [\sym{y_j}]_R):j\in J\}$. Here $[-]_R$ denotes an equivalence class corresponding to the congruence induced by $\syp{R}\cup \sym{R}$. select witnesses canonically by defining for each equivalence class of the form $[\syp{x}]_R$ the element $w(\syp{x})=\bve \{\syp{a}:\syp{a}\in [\syp{x}]_R\}$. We define similarly the element $w(\sym{x})$ for some $x^-\in L^-$. We notice that by our choice of the witnesses we have that the relation $w[R']=\{(w([x]_R),w([y]_R)):([x]_R,[y]_R)\in R'\}$ on $\cpu{L}$ finitary. Then, the quotient $(\cpu{L})/(\syp{R}\cup \syp{R}\cup w[R'])$ is obtained by quotienting $(L^+,L^-,\cpu{L})$ by some finitary congruence $C$ larger than $C_L$. This finitary congruence may be written in canonical form.
\[
C=\bve \{\na(\sy{a^+_i})\cap \na(\sy{a^-_i})\cap \del(\sy{b^+_i})\cap \del(\sy{b^-_i}):i\in I \}
\]
of all congruences of the form $\na(a^+)\cap \na(a^-)\cap \del(b^+)\cap \del(b^-)$ below it. Since the congruence $C_L$ is finitary (by our initial assumption on the biframe $\ca{L}$) and is contained in $C$, we must have that $C_L=\bve _{j\in J}\na(\sy{a^+_j})\cap \na(\sy{a^-_j})\cap \del(\sy{b^+_j})\cap \del(\sy{b^-_j})$ for some $J\se I$. Now, we may write
\[
C=C_L\ve \bve _{i\in I{\sm}J}\na(\sy{a^+_i})\cap \na(\sy{a^-_i})\cap \del(\sy{b^+_i})\cap \del(\sy{b^-_i}).
\]
By Lemma \ref{finitarylemma91}, this is the same as the quotient of $(\cpu{L})/C_L$ by a finitary congruence.

\end{proof}
We call $\mf{S}(\ca{L})$ the ordered collection of biquotients of a finitary biframe $\ca{L}$. We set $\Lq{R}\leq \Lq{S}$ if and only if $S$ induces on $L$ a congruence smaller than that induced by $R$.
From this and from Lemma \ref{finitaryquotient} it follows that the lattice of biquotients of a finitary biframe is anti-isomorphic to the collection of finitary congruences on $L$, ordered under set inclusion. Recall that for a frame $L$ the congruence induced by a relation $R$ is the congruence $\bve \{\na(x)\cap \na(y):(x,y)\in R\}$. We recall that a congruence is finitary if it is induced by inequalities of the form $\syw{x}{x}\leq \syv{y}{y}$ -- that is, by a relation of the form $\{(\syw{x_i}{x_i},\syv{y_i}{y_i}):i\in I\}$. As the map $\na:L\ra \mf{A}(L)$ preserves the frame operations while $\del$ reverses them, this means that the finitary congruences on $L$ are exactly those of the form 
\begin{equation}\label{finitarycongruence}
 \bve_i \na(\syp{x_i})\cap \na(\sym{x_i})\cap\del(\syp{y_i})\cap\del(\sym{y_i}).   
\end{equation}

From this fact alone, one sees that the finitary congruences form a subframe of the ordered collection $\mf{A}(L)$ of all congruences on $L$.
We call $\mf{A}_{\fin}(L)$ the collection of finitary congruences of $L$. We now present several different characterizations of the frame $\mf{A}_{\fin}(L)$. We call $\mf{A}_{\fin}(L)$ the \tc{finitary assembly} of a frame $L$.

\begin{lemma}\label{manysubframes6} For any biframe $\ca{L}$, the following are all descriptions of the finitary assembly of a frame $L$.
\begin{enumerate}
    \item The subframe of $\mf{A}(L)$ generated by all closed congruences of $L$ and the open ones of the form $\del (\sy{x^+}\ve \sy{x^-})$.
    
     \item The subframe of $\mf{A}(L)$ generated by the collection $
    \{\na (\sy{x^+}):x^+\in L^+\}\cup \{\del (\sy{x^-}):x^-\in L^-\}\cup \{\na (\sy{x^-}):x^-\in L^-\}\cup \{\del (\sy{x^+}):x^+\in L^+\}$.
   
\end{enumerate}
\end{lemma}
\begin{proof}
Let us prove each item in turn.
\begin{enumerate}
    \item If we close the collection $\{\del(x):x\in L\}\cup \{\na(\syv{x}{x}):x^+\in L^+,x^-\in L^-\}$ under finite meets we obtain the collection of congruences of the form $\na(\bve_i\syw{x_i}{x_i})\cap \del(\syv{x}{x})=\bve_i \na(\syp{x_i})\cap \na(\sym{x_i})\cap \del(\syp{x})\cap \del(\sym{x})$, and if we close the collection under arbitrary joins we obtain the collection of congruences of the form $\na(\syp{x_i})\cap \na(\sym{x_i})\cap\del(\syp{y_i})\cap\del(\sym{y_i})$.
    \item This is clear from the fact that finitary congruences are precisely those of the form described in \ref{finitarycongruence}.\qedhere
\end{enumerate}
\end{proof}

For a biframe $\ca{L}$, we call a congruence on $L$ a \tc{finitary fitted} one if it is of the form $\bve_i \del(\syv{x_i}{x_i})$. We observe that finitary fitted congruences form a subframe of $\mf{A}_{fin}(L)$. 
Let us now begin our exploration of bispectra of biquotients. Because of the universal property of biquotients, whenever we have a finitary biframe $\ca{L}$ and a finitary quotient $\Lq{R}$, we may view the points of $\Lq{R}$ as the points $\ca{L}\ra \bd{2}$ which factor through the relation $R$. Because of this, the bispectrum of $\Lq{R}$, too, may be viewed as a bisubspace of $\bpt(\ca{L})$. Let us make this precise.
\begin{proposition}\label{spectraofquotients}
Suppose that $\ca{L}$ is a biframe and $\Lq{R}$ a quotient. The bispace $\bpt(\Lq{R})$ is bihomeomorphic to the bisubspace 
\[
\{f\in \bpt(\ca{L}):f\mb{ factors through }R\}\se \bpt(\ca{L}).
\]
\end{proposition}
\begin{proof}
Define the map $\rho:\bpt(\Lq{R})\ra\{f\in \bpt(\ca{L}):f\mb{ factors through }R\}\se \bpt(\ca{L})$ as $\tilde{f}\mapsto f$, that is, as the canonical bijection between these two sets given by the universal property of the quotient. We need to show that these two bispaces have the same bitopology. A positive open of $\bpt(\Lq{R})$ is a set of the form $\{\tilde{f}\in \bpt(\Lq{R}):\tilde{f}^+(a^+)=1\}$ for some $a^+\in L^+$. By the universal property of quotients we have $\tilde{f}^+(a^+)=1$ if and only if $f^+(a^+)=1$ and so the image of a typical positive open is $\{f\in \bpt(\ca{L}):f^+(a^+)=1\}\cap \{f\in \bpt(\ca{L}):f\mb{ factors through }R\}$, which is a typical positive open of $\{f\in \bpt(\ca{L}):f\mb{ factors through }R\}\se \bpt(\ca{L})$. Then, the positive opens of the second bispace are precisely the images under $\rho$ of the positive opens of the first.
\end{proof}

From now on, we adopt the convention of identifying spectra of biquotients $\Lq{R}$ with bisubspaces of $\bpt(\ca{L})$. The first matter we will explore is to study the system of the bispectra of finitary biquotients of a finitary biframe $\ca{L}$, seen as a system of bisubspaces of $\bpt(\ca{L})$. We will see that these bispectra coincide with the patch-closed sets of a certain natural bitopology on $|\bpt(\ca{L})|$. 
\begin{proposition}\label{manyfacts9}
For a finitary biframe $\ca{L}$, we have the following facts.
\begin{enumerate}
    \item $|\bpt(\Lq{\bve_i C_i})|=\bca_i |\bpt(\Lq{C_i})|$ for any collection $C_i$ of finitary congruences on $L$.
    \item $|\bpt(\ca{L}/\del(\syp{a}))|=\va{L}{+}{a}$.
    \item $|\bpt(\Lq{\na(\syp{a})})|=\va{L}{+}{a}^c$.
    \item $|\bpt(\ca{L}/(\less{\syw{x}{x}}{\syv{y}{y}}))|=\va{L}{+}{x}^c\cup \va{L}{-}{x}^c\cu\va{L}{+}{y}\cu\va{L}{-}{y}$.
\end{enumerate}
\end{proposition}
\begin{proof}
Let us prove each item in turn. 
\begin{enumerate}
\item By Proposition \ref{spectraofquotients}, we may make the identification $|\bpt(\ca{L}/\bve_i C_i)|=\{f\in \bpt(\ca{L}):f\mb{ factors through each }C_i\}$ and $\bpt(\ca{L}/C_i)=\{f\in \bpt(\ca{L}):f\mb{ factors through }C_i\}$. The desired result follows.
\item By Proposition \ref{spectraofquotients}, once again we may make the identification $|\bpt(\ca{L}/\del(\syp{a}))|=\{f\in \bpt(\ca{L}):f\mb{ factors through }\del(\syp{a})\}$, and by definition of open congruences this is the same as $\{f\in \bpt(\ca{L}):f^+(a^+)=1\}$, which exactly the open $\va{L}{+}{a}$.
\item The argument is similar to that of item (2); here we use the fact that a map $f:\ca{L}\ra \bd{2}$ factors through a closed congruence $\na(\syp{a})$ if and only if $f^+(a^+)=0$.
\item The argument is again similar; here we use the fact that a map $f:\ca{L}\ra \bd{2}$ factors through $\less{\syw{x}{x}}{\syv{y}{y}}$ if and only if we have $f(\syw{x}{x})\leq f(\syv{y}{y})$, and that this holds if and only if we have either $f(\syw{x}{x})=0$ or $f(\syv{y}{y})=1$, and that this holds if and only if we have either $f^+(x^+)=0$, or $f^-(x^-)=0$, or $f^+(y^+)=1$, or $f^-(y^-)=1$.\qedhere
\end{enumerate}
\end{proof}
We are now ready to describe the bispectrum of an arbitrary biquotient.
\begin{proposition}
Suppose that $\ca{L}$ is a finitary biframe and that $C$ is a finitary congruence on $L$. Then, the underlying set of $\bpt(\Lq{C})$ is 
\[
\bca\{\va{L}{+}{a}^c\cup \va{L}{+}{b}\cup \va{L}{-}{a}^c\cup \va{L}{-}{b}:\less{a^+}{b^+}\cap \less{a^-}{b^-}\se C\}
\]
\end{proposition}
\begin{proof}
An arbitrary finitary congruence $C$ on $L$ is of the form $\bve \{ \na(\syp{x}\we\sym{x})\cap\del(\syp{y}\ve\sym{y}):\na(\syp{x})\cap \na(\sym{x})\cap\del(\syp{y})\cap\del(\sym{y})\se C\}$. The result follows from this fact and from items (1) and (4) of Proposition \ref{manyfacts9}. 
\end{proof}

In \cite{suarez2020category}, the following fact is shown.

\begin{lemma}\label{bisobersarequotients}
For a finitary biframe $\ca{L}$, any bisober bisubspace inclusion $\bpt(\ca{M})\inclu \bpt(\ca{L})$ is, up to bihomeomorphism, and inclusion of the form $\bpt(\Lq{C})\se \bpt(\ca{L})$ for some finitary congruence $C$ on $L$. 
\end{lemma}

This gives us a way of describing an arbitrary bisober bisubspace of a spectrum $\bpt(\ca{L})$ for some finitary biframe $\ca{L}$.
\begin{corollary}\label{generalbisobersub}
If $\ca{L}$ is a finitary biframe, the bisober bisubspaces of $\bpt(\ca{L})$ are exactly those whose underlying sets are of the form 
\[
\bca_i \va{L}{+}{a_i}^c\cup \va{L}{+}{b_i}\cup \va{L}{-}{a_i}^c\cup \va{L}{-}{b_i}
\]
for some $a^+,b^+\in L^+$ and some $a^-,b^-\in L^-$.
\end{corollary}
\begin{proof}
The bisober bisubspaces of $\bpt(\ca{L})$, by Lemma \ref{bisobersarequotients}, may be identified with the bisubspaces of $\bpt(\ca{L})$ of the form $\bpt(\Lq{C})$ for some finitary congruence $C$ on $L$. By item (4) of Proposition \ref{manyfacts9}, we know that all bisober bisubspaces are of the form described in the statement. For the converse, we observe that by the same item every set of the form $\bca_i \va{L}{+}{a_i}^c\cup \va{L}{+}{b_i}\cup \va{L}{-}{a_i}^c\cup \va{L}{-}{b_i}$ is the bispectrum of the quotient of $\ca{L}$ by 
\[
\bve_i \na(\syp{x_i})\cap \na(\sym{x_i})\cap\del(\syp{y_i})\cap\del(\sym{y_i}).\qedhere
\]
\end{proof}
The bisober bisubspaces of a bisober bispace $\bpt(\ca{L})$, then, are a subcoframe of the powerset $\ca{P}(\bpt(\ca{L}))$. Thus, their underlying sets are precisely the closed sets of some topology on $|\bpt(\ca{L})|$. In the next section we shall see that this topology is the patch topology of the bitopological analogue of the Skula space of a space.

\section{Skula bispaces}

We introduce bitopological analogues of the Skula topology for a space. Every bispace $X$ determines two topological spaces, namely $(X,\Om^+(X))$ and $(X,\Om^-(X))$. Each of these spaces has a corresponding Skula space, we call these the \textit{positive} Skula space of $X$, and the \textit{negative} Skula space, respectively.

\begin{itemize}
    \item We define the \textit{Skula bispace} of a bispace $X$ to be the space $Sk(X)$ such that its collection of points is $|X|$, the positive opens are the topology generated by the positive opens and the negative closed sets of $X$, and the negative opens are the topology generated by the negative opens and the positive closed sets of $X$.
    \item We define the \textit{closed-fitted} Skula bispace the bispace $Sk_{cf}(X)$ with $|X|$ as an underlying set of points, with positive opens the patch-opens of $X$, and negative opens the topology generated by the positive and the negative closed sets of $X$.
    \item We define the \textit{positive-negative} Skula bispace as the bispace $Sk_{\pm}(X)$ with $|X|$ as its underlying set of points, with positive opens the positive Skula space, and negative opens the negative Skula space.
\end{itemize}

We observe right away that all three Skula bispaces of a bispace $X$ have the same patch, i.e. the topology on $|X|$ generated by the positive and the negative opens together with the positive and the negative closed sets. For a bispace $X$, we say that a patch-open set is \tc{finitary} is it is of the form $(U_1^+\cap U^-_1)\cup ...\cup (U^+_n\cap U^-_n)$. By analogy with pointfree topology, we call an intersection of patch-open sets a \tc{fitted} subset of $X$. Finally, we say that a fitted subset of $X$ is \tc{finitary fitted} if it is an intersection of finitary patch opens. We observe that the finitary fitted subsets are exactly those of the form $\bca_iU^+_i\cup U^-_i$. We also observe the following fact.
\begin{lemma}
For a bispace $X$, the negative open sets of the closed-fitted Skula bispace $Sk_{cf}$ are exactly the complements of the finitary fitted sets of $X$.
\end{lemma}

The terminology ``closed-fitted" will become clearer once, in the next sections, we look at pointfree versions of all three these bispaces.

\begin{theorem}\label{skulabisobers}
For any finitary biframe $\ca{L}$, the bisober bisubspaces of $\bpt(\ca{L})$ are exactly those such that their underlying sets are patch-closed sets of $Sk(\bpt(\ca{L}))$, and these coincide with the patch-closed sets of $Sk_{cf}(\bpt(\ca{L}))$ and of $Sk_{\pm}(\bpt(\ca{L}))$.
\end{theorem}
\begin{proof}
This follows from Corollary \ref{generalbisobersub} and from the definition of the three Skula bispaces.
\end{proof}

Let us refine the result above in order to understand the meaning of the three different bitopologies on this patch that these bispaces represent.

\begin{lemma}\label{reallemmacf9}
For a finitary biframe $\ca{L}$, the patch-closed bisubspaces of $\bpt(\ca{L})$ are exactly those of the form $\bpt(\Lq{\na(x)})$ for some $x\in L$. The finitary patch-fitted bisubspaces correspond to those of the form $\bca_i \bpt(\Lq{C})$ for some finitary fitted congruence $C$.  
\end{lemma}
\begin{proof}
An arbitrary patch-closed bisubspace of $\bpt(\ca{L})$ is of the form $\bca_i \va{L}{+}{a_i}^c\cup \va{L}{-}{a_i}^c$. By items (1) and (4) of Proposition \ref{manyfacts9}, then, the patch-closed sets are exactly those of the form $|\bpt(\ca{L}/(\bve_i \na(\syp{a_i})\cap \na(\sym{a_i})))|=|\bpt(\ca{L}/\na(\bve_i\syp{a_i}\we \sym{a_i}))|$ for some family $a^+_i\in L^+$ and some family $a^-_i\in L^-$. These, then, are precisely the bispectra of quotients of the form $\na(x)$ for $x\in L$. Similarly, an arbitrary finitary patch-fitted set is of the form $\bca_i \va{L}{+}{a_i}\cup \va{L}{-}{a_i}$, and by items (1) and (4) of Proposition \ref{manyfacts9}, this is equal to $|\bpt(\ca{L}/\bve_i(\del(\syp{a_i})\cap \del(\sym{a_i})))|=|\bpt(\ca{L}/(\bve_i \del(\syv{a_i}{a_i})))|$. Since $\bve_i\del(\syv{a_i}{a_i})$ is the general form of a finitary fitted congruence, we are done.
\end{proof}

\begin{proposition}
For a finitary biframe $\ca{L}$, the closed-fitted Skula bispace $Sk_{cf}(\ca{L})$ is such that 
\begin{itemize}
    \item its positive closed sets are the underlying sets of the bisober of the form $\bpt(\Lq{\na(x)})$ for some $x\in L$;
    \item its negative closed ones are bisober bisubspaces of the form $\bpt(\Lq{C})$ for some finitary fitted congruence $C$. 
\end{itemize}
\end{proposition}
\begin{proof}
The positive closed sets of $Sk_{cf}(\bpt(\ca{L}))$ are the patch-closed sets of $\bpt(\ca{L})$, and by \ref{reallemmacf9} these are the same as the bisober bisubspaces of the form $\bpt(\ca{L}/\na(x))$. The negative closed sets of $Sk_{cf}(\ca{L})$ are the finitary patch-fitted subsets of $\bpt(\ca{L})$, and these by Lemma \ref{reallemmacf9} are the bisober bisubspaces of the form $\bca_i \bpt(\Lq{C})$ for some finitary fitted congruence $C$. 
\end{proof}

Let us move to the positive-negative Skula bispace. To be able conceptualize better the analogous result for this Skula bispace, we need a new notion. For a biframe $\ca{L}$, among the congruences of $L$, there exist special ones, in that they are generated by relations involving only the positive or only the negative generators. For a biframe $\ca{L}$, and for a congruence $C^+$ on $L^+$, we call $\syp{C}$ the congruence on $L$ generated by $\{(\syp{a},\syp{a}):[a^+]_{C^+}\leq [b^+]_{C^+}\}$. Notice that this is the same as the congruence $\bve \{\na(\syp{a})\cap \del(\syp{b}):[a^+]_{C^+}\leq [b^+]_{C^+}\}$. We refer to a congruence on $L$ as \tc{positive} if it is of the form $\syp{C}$ for some congruence $C^+$ on $L^+$. For a congruence $C^-$ on $L^-$, we define the congruence $\sym{C}$ on $L$ similarly, and we define \tc{negative} congruences on $L$ similarly as the positive ones. For a finitary biframe $\ca{L}$, we define a bisubspace of $\bpt(\ca{L})$ to be \tc{positive} if it is of the form $\bpt(\Lq{\sy{C^+}})$ for some positive congruence $\syp{C}$. We notice that positive bisober bisubspaces $Y\se \bpt(\ca{L})$ are such that $\p{f}{f_1}\in Y$ and $\p{f}{f_2}\in \bpt(\ca{L})$ implies that $\p{f}{f_2}\in Y$. That is, positive bisober bisubspaces ``stretch across all possible negative coordinates". We define negative bisubspaces of $\bpt(\ca{L})$ similarly.

\begin{proposition}
For a finitary biframe $\ca{L}$, the positive-negative Skula bispace $Sk_{\pm}(\bpt(\ca{L}))$ is such that 
\begin{itemize}
    \item its positive closed sets are the underlying sets of the bisober positive bisubspaces -- that is -- those of the form $\bpt(\Lq{\sy{C^+}})$ for some congruence $C^+$ on $L^+$;
    \item its negative closed sets correspond to the underlying sets of the bisober negative bisubspaces -- that is -- those of the form $\bpt(\Lq{\sy{C^-}})$ for some congruence $C^-$ on $L^-$.
\end{itemize}
\end{proposition}
\begin{proof}
The positive closed sets of $Sk_{\pm}(\bpt(\ca{L}))$ are the subsets of $\bpt(\ca{L})$ of the form $\bca_i \va{}{}{}\cup \va{}{}{}$, and by \ref{reallemmacf9} these are the same as the bisober bisubspaces of the form $\bpt(\ca{L}/\syp{C}))$. The argument for the negative topology is similar. 
\end{proof}

\section{A visual summary}

We now give a visual interpretation of the results that we have gathered so far. For any frame $L$ we have its assembly $\mf{A}(L)$. If we take $L$ to be linear, then we may visualize its spectrum as a line, where opens are left closed segments. The assembly $\mf{A}(L)$ has a spectrum which may be viewed as a bitopological space with $|\pt(L)|$ as its underlying set of points. Its positive opens are just the opens of the original topology, and its negative opens are the topology generated by the closed sets of $\pt(L)$. Below we depict a typical positive and a typical negative open of $Sk(\pt(L))$, seen as a bitopology.

 \includegraphics[scale=0.78]{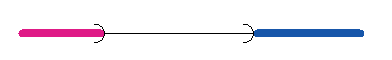}

These two topologies ``generate" all the sober subspaces of $L$ in the sense that all sober subspaces are patch-closed in this bitopology. There is no other obvious way of dividing the topology generated by these sets (namely $Sk(\pt(L))$) into two topologies.

When we are dealing with a pointfree bispace $\ca{L}$, instead, look at the topology in which the closed sets coincide with the underlying sets of the bisober bisubspaces. As we have seen, this is the topology generated by the positive opens, together with the positive closed sets, together with the negative opens, together with the negative closed sets. To visualize this let us assume that $\ca{L}$ is a biframe in which both frames are linear. The space $\bpt(\ca{L})$ will then be a rectangle, possibly with some parts missing due to the inequalities that hold in $L$. Positive opens are left portions of the partial rectangle, and negative opens are lower portions. Then, these four kinds of sets correspond to four kinds of rectangles. Suppose that our set of types of rectangles is $\{O^+,O^-,C^+,C^-\}$, for ``positive opens, negative opens, positive closed sets, negative closed sets". 

\medskip

\includegraphics[scale=0.66]{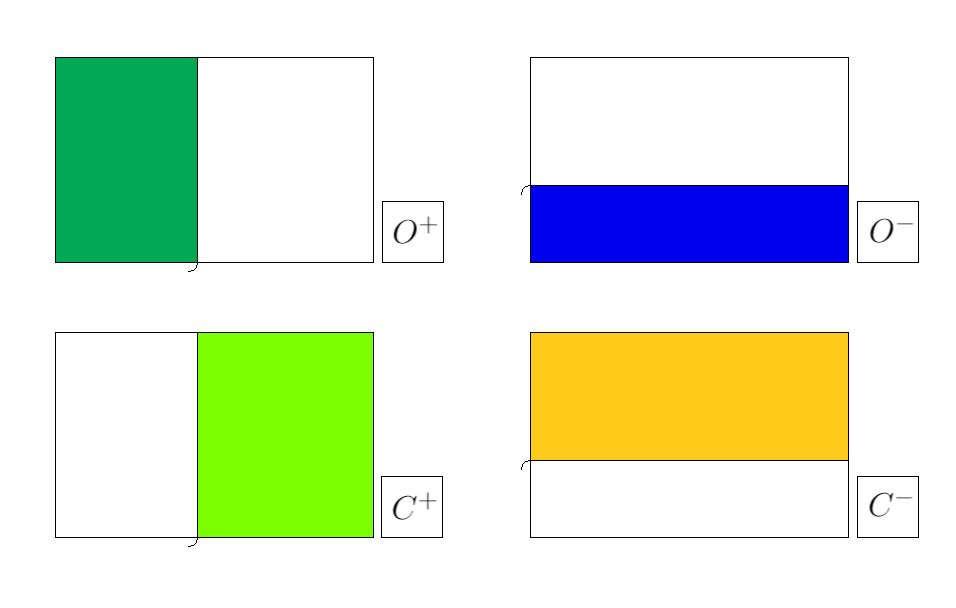}

There are only three ways of partitioning a set of four elements $\{a,b,c,d\}$ into two sets of two, namely $\{\{a,b\},\{c,d\}\}$, $\{\{a,c\},\{b,d\}\}$, and $\{\{a,d\},\{b,c\}\}$. Each such partition of the set $\{O^+,O^-,C^+,C^-\}$ determines a bitopological space whose patch is the topology generated by $O^+\cup O^-\cup C^+\cup C^-$, in the sense that a partition $\{\{x_1,x_2\},\{x_3,x_4\}\}$ determines the bitopological space in which the positive topology is the one generated by the collection $x_1\cup x_2$, and the negative one the one generated by the collection $x_3\cup x_4$. The three partitions of $\{O^+,O^-,C^+,C^-\}$ into two sets of cardinality two correspond to the three Skula bispaces. Let us summarize the information that we have gathered about each of these.

\begin{itemize}
    \item The partition $\{\{O^+,O^-\},\{C^+,C^-\}\}$ yields the closed-fitted Skula bispace.

\begin{figure}[H]
  \centering
  \begin{minipage}[b]{0.4\textwidth}
  \includegraphics[scale=0.78]{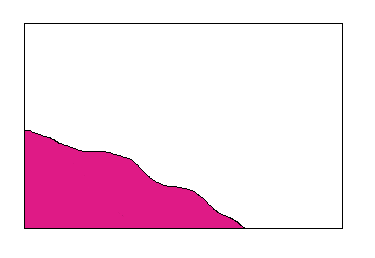}
     \caption{\small{A typical positive open of $Sk_{cf}(\bpt(\ca{L}))$ is a lower left portion of the partial rectangle $\bpt(\ca{L})$. These sets also coincide with the patch-open sets of $\bpt(\ca{L})$, and the complements $|\bpt(\Lq{\na(x)})|^c$ for some $x\in L$.}}
  \end{minipage}
  \hfill
  \begin{minipage}[b]{0.4\textwidth}
     \includegraphics[scale=0.78]{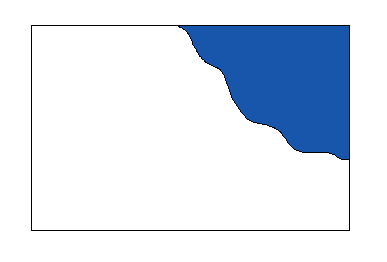}
  \caption{\small{A typical negative open of $Sk_{cf}(\bpt(\ca{L}))$ is an upper right portion of $\bpt(\ca{L})$. These sets also coincide with the complements of the finitary patch-fitted sets of $\bpt(\ca{L})$, and the complements $|\bpt(\Lq{C})|^c$ for a finitary fitted congruence $C$ on $L$.}}
  \end{minipage}
\end{figure}

\item The partition $\{\{O^+,C^+\},\{O^-,C^-\}\}$ yields the positive-negative Skula bispace. 

\begin{figure}[H]
  \centering
  \begin{minipage}[b]{0.4\textwidth}
  \includegraphics[scale=0.78]{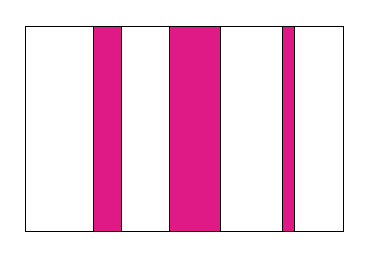}
     \caption{\small{A typical positive open of $Sk_{\pm}(\bpt(\ca{L}))$ is a collection of vertical portions of the partial rectangle $\bpt(\ca{L})$. These sets also coincide with the complements of the d-sober positive bisubspaces of $\bpt(\ca{L})$, and the complements $|\bpt(\Lq{\syp{C}{}})|^c$ for a congruence $C^+$ on $L^+$.}}
  \end{minipage}
  \hfill
  \begin{minipage}[b]{0.4\textwidth}
 \includegraphics[scale=0.78]{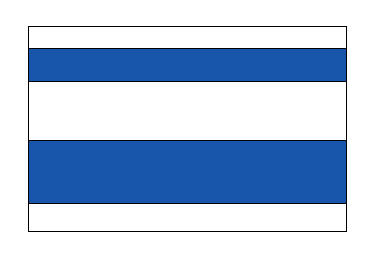}
     \caption{\small{A typical negative open of $Sk_{\pm}(\bpt(\ca{L}))$ is a collection of horizontal portions of the partial rectangle $\bpt(\ca{L})$. These sets also coincide with the complements of the d-sober negative bisubspaces of $\bpt(\ca{L})$, and the complements $|\bpt(\Lq{\sym{C}{}})|^c$ for a congruence $C^-$ on $L^-$.}}
  \end{minipage}
\end{figure}

\item The partition $\{\{O^+,C^-\},\{O^-,C^+\}\}$ yields the Skula bispace. 

\begin{figure}[H]
  \centering
  \begin{minipage}[b]{0.4\textwidth}
  \includegraphics[scale=0.78]{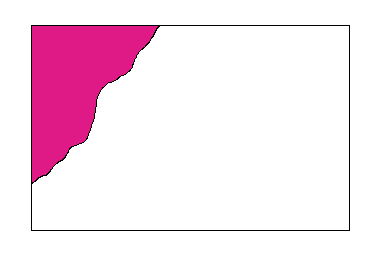}
     \caption{\small{A typical positive open of $Sk(\bpt(\ca{L}))$ is an upper left portion of the partial rectangle $\bpt(\ca{L})$. These sets also coincide with those of the topology generated by the positive opens and the negative closed sets of $\bpt(\ca{L})$.}}
  \end{minipage}
  \hfill
  \begin{minipage}[b]{0.4\textwidth}
\includegraphics[scale=0.78]{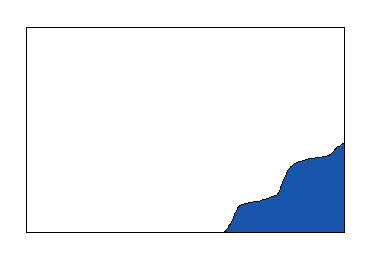}
     \caption{\small{A typical positive open of $Sk(\bpt(\ca{L}))$ is a lower right portion of the partial rectangle $\bpt(\ca{L})$. These sets also coincide with those of the topology generated by the positive opens and the negative closed sets of $\bpt(\ca{L})$.}}
  \end{minipage}
\end{figure}
\end{itemize}

\section{The two variations of the assembly of a finitary biframe}
\subsection{The assembly functor and its spatial counterpart}

In this subsection, we recall some facts about the assembly of a finitary biframe. Recall that for a biframe $\ca{L}$ we define its \tc{assembly} to be the structure
\[
\mf{A}(\ca{L})=(\mf{A}_{\na^+\del^-}(L),\mf{A}_{\na^-\del^+}(L),\mf{A}_{\fin}(L))
\]
where $\mf{A}_{\na^+\del^-}(L)$ is the subframe of $\mf{A}_{\fin}(L)$ generated by the collection $\{\na(a^+):a^+\in L^+\}\cup \{\del(a^-):a^-\in L^-\}$ and $\mf{A}_{\na^-\del^+}(L)$ is defined similarly. This structure is introduced in \cite{schauerte92}, where it is also shown that there is a canonical biframe embedding $\na=\pa{\na}:\ca{L}\ra \mf{A}(\ca{L})$ assigning to each $a^+$ the corresponding closed congruence $\na(a^+)$, and defined similarly on the negative component; and that this has a certain universal property. Let us recall that for a biframe $\ca{L}$ we have the notion of \tc{bipseudocomplement}. The bicomplement of a generator $x^+\in L^+$ is denoted as $\sim x^+$, and it is $\bve \{x^-\in L^-:\syw{x}{x}=0\mb{ in }L\}$. Bipseudocomplements of negative generators are defined similarly. When a bipseudocomplement $\sim x^+$ is such that $\syv{x}{x}=1$ in $L$, the element $\sim x^+$ is called a \tc{bicomplement} of $x^+$. In \cite{schauerte92} it is shown the assembly of a biframe $\ca{L}$ has the universal property that whenever we have a biframe map $f:\ca{L}\ra \ca{M}$ such that it provides complements for all elements of $L^+\cup L^-$ there is a unique biframe map $\tilde{f}:\Ad\ra \ca{M}$ making the following diagram commute.
\[
\begin{tikzcd}
{} 
& \Ad
\arrow{dr}{\tilde{f}} \\
\ca{L}
\arrow{ur}{\na} 
\arrow{rr}{f} 
&& \ca{M}
\end{tikzcd}
\]

We now recall that the assembly of a biframe is isomorphic to a certain free construction. From now on, we will abbreviate the filter completion functor $\mf{Filt}:\bd{Distr}\ra \bd{Frm}$ as simply $\mf{F}$. In \cite{suarez2020category} it is shown that the assembly of a biframe $\ca{L}$ is isomorphic to the biframe quotient
\renewcommand{\Fi}{\mf{F}}
\renewcommand{\Fip}{\mf{F}(L^+)}
\renewcommand{\Fim}{\mf{F}(L^-)}
\[
(L^+\oplus \Fim,L^-\oplus \Fip,L^+\oplus \Fim\oplus L^-\oplus \Fip)/(C_L\cup \mc{Com}^+_L\cup \mc{Com}^-_L),
\]
where name ``$\mc{Com}^+_L$" stands for ``complementation", and denotes the relation 
\begin{align*}
    & \{(\syp{a}\we \syp{\up a},0),(1,\syp{a}\ve\syp{\up a}):a^+\in L^+\}
\end{align*}
and the relation $\mc{Com}^-_L$ is defined similarly for negative generators $a^-\in L^-$. We will denote as $\mf{A}^m(\ca{L})$ the biframe $(L^+\oplus \Fim,L^-\oplus \Fip,L^+\oplus \Fim\oplus L^-\oplus \Fip)$. Above, we have slightly abused notation by regarding $C_L$ as a congruence on $\mf{A}^m(\ca{L})$ rather than on $\cpu{L}$. Let us state this as a proposition.

\begin{proposition}\label{assemblyiso}
For a finitary biframe $\ca{L}$ its assembly is isomorphic to
\[
(L^+\oplus \Fim,L^-\oplus \Fip,L^+\oplus \Fim\oplus L^-\oplus \Fip)/(C_L\cup \mc{Com}^+_L\cup \mc{Com}^-_L).
\]
\end{proposition}

By Lemma \ref{finitarylemma9}, the assembly of any finitary biframe is finitary. Thus, we have a well-defined assignment $\mf{A}:\mf{Obj}(\bd{BiFrm_{fin}})\ra \mf{Obj}(\bd{BiFrm_{fin}})$. We can extend this construction to an endofunctor $\mf{A}:\bd{BiFrm_{fin}}\ra \bd{BiFrm_{fin}}$. Recall that, since $\mf{A}:\bd{Frm}\ra \bd{Frm}$ is an endofunctor, whenever we have a frame map $f:L\ra M$ we always have a frame map $\mf{A}(f):\mf{A}(L)\ra \mf{A}(M)$ defined as $\na(x)\mapsto \na(f(x))$ and as $\del(x)\mapsto \del(f(X))$ on generators. We also observe that for a frame map $f:L\ra M$ the map $\mf{A}(f)$ sends finitary congruences to finitary congruences, and so the map $\mf{A}(f):\mf{A}_{\fin}(L)\ra \mf{A}_{\fin}(M)$ is well-defined. For a morphism $f:\ca{L}\ra \ca{M}$ of finitary biframes we define the map $\mf{A}(f):\mf{A}(\ca{L})\ra \mf{A}(\ca{M})$ as the biframe map whose main component is $\mf{A}(f):\mf{A}_{\fin}(L)\ra \mf{A}_{\fin}(M)$.

In \cite{suarez2020category} it is then shown that the assembly functor is the pointfree counterpart of the Skula functor.

\begin{theorem}
There is a natural isomorphism $\alpha:Sk\circ \bpt\ra \bpt\circ \mf{A}$.
\end{theorem}

With the following theorem, we show the importance of the assembly construction for the theory of finitary biframes (see \cite{suarez2020category}).

\begin{theorem}
For any finitary biframe $\ca{L}$, we have the following.
\begin{itemize}
    \item The biframe $\mf{A}(\ca{L})$ has the universal property that it provides bicomplements freely to all elements of $L^+\cup L^-$. 
    \item The main component of $\Ad$ is anti-isomorphic to the coframe $\mf{S}(\ca{L})$ of all biquotients of $\ca{L}$.
    \item There is a bihomeomorphism $\alpha_{\ca{L}}:Sk(\bpt(\ca{L}))\cong \bpt(\Ad)$
    \item The patch-closed sets of $\bpt(\Ad)$, under the bijection $|\bpt(\ca{L})|\cong|\bpt(\mf{A}(\ca{L}))|$, coincide with the underlying sets of the bisober bisubspaces of $\bpt(\ca{L})$.
\end{itemize}
\end{theorem}

The assembly, then, plays the role of a ``bispace of all bisubspaces" of a finitary biframe. However, the bitopological structure of $\mf{A}(\ca{L})$ is somewhat unnatural. For a frame $L$, the assembly has a natural bitopological structure: we have that $\mf{A}(L)$ is generated by the subframe $\mf{Cl}(L)$ of closed congruences and that $\mf{Fitt}(L)$ of fitted congruences. For a finitary biframe, we shall see that there is more than one way of dividing $\mf{A}_{\fin}(L)$ naturally into two generating topologies. This will give raise to two variations of the assembly of a finitary biframe.

\subsection{The $\mf{A}_{cf}$ functor and its spatial counterpart}

Recall that for a bispace $X$ its \tc{closed-fitted Skula bispace} is the bispace $Sk_{cf}(X)$ such that it has the same underlying set of points as $X$, whose positive topology is the one generated by the positive and the negative opens of $X$ (the patch topology of $X$), and whose negative topology is that generated by the positive and the negative closed sets of $X$. We now define the pointfree version of this functor, taking the duality of finitary biframe as our pointfree setting. For a finitary biframe $\ca{L}$, we define its \tc{closed-fitted assembly} as the biframe
\[
\Acf=(\mf{Cl}(L),\mf{Fitt}_{\fin}(L),\mf{A}_{\fin}(L)).
\]
Here $\mf{Cl}(L)$ is the ordered collection of all closed congruences of $L$, while $\mf{Fitt}_{\fin}(L)$ is the ordered collection of the finitary fitted ones (which, recall, are the congruences of the form $\bve_i \del(\syv{x_i}{x_i})$). This triple indeed is a biframe, by item (1) of Lemma \ref{manysubframes6}.

\begin{proposition}\label{assemblyisocf}
For a finitary biframe $\ca{L}$ we have an isomorphism of biframes
\[
\Acf\cong (L^+\oplus L^-,\Fip\oplus \Fim,L^+\oplus \Fim\oplus L^-\oplus \Fip)/(C_L\cup \mc{Com}^+_L\cup \mc{Com}^-_L).
\]
\end{proposition}
\begin{proof}
The patch of the free biframe in the statement is isomorphic to $\mf{A}_{\fin}(L)$, by Proposition \ref{assemblyiso}. The isomorphism given by this proposition acts as $\sy{\syp{x}}\mapsto \na(\syp{x})$ and $\sy{\sym{x}}\mapsto \na(\sym{x})$. This means that under this isomorphism the frame $\cpu{L}$ is mapped to the subframe of $\mf{A}_{\fin}(L)$ generated by the elements of the form $\na(\syp{x})$ together with those of the form $\na(\sym{x})$. This is the frame $\mf{Cl}(L)$. We notice that under the isomorphism of Proposition\ref{assemblyiso} we also have $\sy{\syp{x}}\mapsto \del(\syp{x})$, and similarly for generators $x^-\in L^-$. This means that the image of the frame $\Fip\oplus \Fim$ under this isomorphism is the subframe of $\mf{A}_{\fin}(L)$ generated by the elements of the form $\del(\syp{x})$ together with those of the form $\del(\sym{x})$; and this is the frame $\mf{Fitt}_{\fin}(L)$.
\end{proof}

\begin{proposition}
The closed-fitted assembly of a finitary biframe is finitary. 
\end{proposition}
\begin{proof}
This follows from Proposition \ref{assemblyisocf} and by Lemma \ref{finitarylemma9}.
\end{proof}

For a biframe map $f:\ca{L}\ra \ca{M}$, we have that the frame map $f:L\ra M$ induces a frame map $\mf{A}(f):\mf{A}(L)\ra \mf{A}(M)$, by functoriality of the classical assembly construction. We define the biframe map $\mf{A}_{cf}:\Acf\ra \mf{A}_{cf}(\ca{M})$ as the pairing of the frame maps obtained by restricting $\mf{A}(f)$ to $\mf{Cl}(L)$ and to $\mf{Fitt}_{\fin}(L)$, respectively. Let us show that this is a well-defined biframe map.

\begin{lemma}
For a frame map $f:L\ra M$, the map $\mf{A}(f):\mf{A}(L)\ra \mf{A}(M)$ maps closed congruences to closed ones and finitary fitted ones to finitary fitted ones.
\end{lemma}
\begin{proof}
The map $\mf{A}(f)$ is defined on closed congruences as $\na(x)\mapsto \na(f(x))$, and so indeed it maps closed congruences to closed ones. It also acts on open congruences similarly, as $\del(x)\mapsto \del(f(x))$. Then, a finitary fitted congruence $\bve_i \del(\syp{x_i}\we \sym{x_i})$ is mapped by this to $\bve_i\del(f(\syp{x_i}\we \sym{x_i}))=\bve_i \del(\sy{f^+(x^+_i)}\we \sy{f^-(x^-_i)})$, which is a finitary fitted congruence of $M$. 
\end{proof}

By its definition, the assignment $\mf{A}_{cf}:\mf{Mor}(\bd{BiFrm_{fin}})\ra \mf{Mor}(\bd{BiFrm_{fin}})$ also respects identities and compositions. We have proved the following.

\begin{proposition}
The assignment $\mf{A}_{cf}:\bd{BiFrm_{fin}}\ra \bd{BiFrm_{fin}}$ is a functor.
\end{proposition}

Let us now compare the bispectrum of a biframe with the spectrum of its main component. In the following proof, for a biframe map $\pa{f}:\ca{L}\ra \ca{M}$ we call $f:L\ra M$ the frame map between the main frame components of $\ca{L}$ and $\ca{M}$ generated by $f^+$ and $f^-$.
\begin{lemma}\label{gothroughthepatch}
For every biframe $\ca{L}$, there is a bijection $p_{\ca{L}}:|\bpt(\ca{L})|\cong |\pt(\pi_3(\ca{L}))|$ given by $\pa{f}\mapsto f$. Its inverse maps $f:L\ra M$ to the pairing of its restrictions to $\syp{L}$ and $\sym{L}$.
\end{lemma}
\begin{proof}
By definition of biframe map, we have that $\pa{f}:\ca{L}\ra 2$ is a bipoint if and only if there is a frame map $f:L\ra M$ such that it is generated by $f^+$ and $f^-$. The assignment is then well-defined. To see that this is injective, we notice that if we have two distinct frame maps $f_1,f_2:L\ra M$, since $L$ is generated by $\syp{L}$ and $\sym{L}$ the two maps must differ on their restriction to at least one of these frames. For surjectivity, we notice that any frame map $f:L\ra M$ is that generated by its restriction to $\syp{L}$ and $\sym{L}$, as these two frames generate $L$.
\end{proof}

\begin{lemma}\label{twobiframessamepoints}
Suppose that $\ca{L}$ and $\ca{L}'$ are two biframes with the same main component $L$. Then, we have a bijection $\bpt(\ca{L})\cong \bpt(\ca{L}')$.
\end{lemma}
\begin{proof}
The map that we have described in the statement is the composition of the bijections $p_{\ca{L}}$ and $p^{-1}_{\ca{L}'}$, as defined in Lemma \ref{gothroughthepatch} above.
\end{proof}

Let us now look at the relation between the bispectrum of a biframe $\ca{L}$ and that of its closed-fitted assembly $\mf{A}_{cf}(\ca{L})$.  Recall that every point of $\Ad$ is of the form $\tilde{f}$ for some bipoint $f:\ca{L}\ra \bd{2}$, where $\tilde{f}$ is the unique biframe map making the following diagram commute.
\[
\begin{tikzcd}
{} 
& \Ad
\arrow{dr}{\tilde{f}} \\
\ca{L}
\arrow{ur}{\na} 
\arrow{rr}{f} 
&& \bd{2}
\end{tikzcd}
\]
Following the proof of Lemma \ref{twobiframessamepoints}, we obtain that we have a bijection $\bpt(\Ad)\cong \bpt(\Acf)$ given by $\tilde{f}\mapsto \tilde{f}_{cf}$, where $\tilde{f}$ is the pairing of the restrictions of the frame map $f:\mf{A}_{\fin}(\ca{L})\ra 2$ to the two subframes $\mf{Cl}(L)$ and $\mf{Fitt}_{\fin}(L)$. Recall that $\mf{Cl}(L)$ is generated by the congruences of the form $\na(\syp{x})$ together with those of the form $\na(\sym{x})$, and that $\mf{Fitt}_{\fin}(L)$ is generated by those of the form $\del(\syp{x})$ together with those of the form $\del(\sym{x})$. Then, any bipoint $\tilde{f}_{cf}:\Acf\ra \bd{2}$ is completely determined by its action on these subbasic congruences. We have proved the following fact.

\begin{proposition}\label{bijectionpointscf}
The assignment $\beta_{\ca{L}}:f\mapsto \tilde{f}_{cf}$, with this map defined as
\begin{align*}
    & \tilde{f}_{cf}:\Acf\ra \bd{2}\\
    & \na(\syp{x})\mapsto f^+(x^+)\\
    & \del(\syp{x})\mapsto \neg f^+(x^+),
\end{align*}
and similarly on the congruences of the form $\na(\sym{x})$ and $\del(\sym {x})$, constitutes a bijection $|\bpt(\ca{L})|\cong |\bpt(\mf{A}_{cf}(\ca{L}))|$.
\end{proposition}

We now work towards showing that for every finitary biframe $\ca{L}$ the bijection $\beta_{\ca{L}}:\bpt(\ca{L})\cong \bpt(\Acf)$ constitutes a bihomeomorphism $Sk_{cf}(\bpt(\ca{L}))\cong \bpt(\Acf)$.

\begin{lemma}\label{lemmabihomeo9cf}
For any finitary biframe $\ca{L}$, we have the following.
\begin{itemize}
    \item $\beta_{\ca{L}}[\va{L}{+}{x}]=\varphi_{\Acf}^+(\na(\syp{x}))$,
    \item $\beta_{\ca{L}}[\va{L}{-}{x}]=\varphi_{\Acf}^+(\na(\sym{x}))$,
    \item $\beta_{\ca{L}}[\va{L}{+}{x}^c]=\varphi_{\Acf}^-(\del(\syp{x}))$,
    \item $\beta_{\ca{L}}[\va{L}{-}{x}^c]=\varphi_{\Acf}^-(\del(\sym{x}))$.
\end{itemize}
\end{lemma}
\begin{proof}
let us show the first and the third item.
\begin{itemize}
    \item For $x^+\in L^+$, we have the following chain of equalities. For the equality between the second and the third line, we have used the fact that by Proposition \ref{bijectionpointscf} we know that $\beta_{\ca{L}}:f\mapsto \tilde{f}_{cf}$ is a bijection between the points of $\ca{L}$ and those of $\Acf$, and so it is surjective.  For the equality after that, we have used the definition of $\beta_{\ca{L}}$ given in Proposition \ref{bijectionpointscf}.
    \begin{align*}
        & \beta_{\ca{L}}[\va{L}{+}{x}]=\\
        & =\{\tilde{f}_{cf}\in \bpt(\Acf): f\in \bpt(\ca{L}),f^+(x^+)=1\}=\\
        & =\{\tilde{f}_{cf}\in \bpt(\Acf):\tilde{f}_{cf}(\na(\syp{x}))=1\}=\\
        & =\varphi^{+}_{\Acf}(\na(\syp{x})).
    \end{align*}
    \item For $x^+\in L^+$, we have the following chain of equalities. For the equality between the second and the third line, we have used the fact that by Lemma \ref{bijectionpointscf} we know that $\beta_{\ca{L}}:f\mapsto \tilde{f}_{cf}$ is a bijection between the points of $\ca{L}$ and those of $\Acf$, and so it is surjective. We have also used the definition of $\beta_{\ca{L}}$.
    \begin{align*}
        & \beta_{\ca{L}}[\va{L}{+}{x}^c]=\\
        & =\{\tilde{f}_{cf}\in \bpt(\Acf): f\in \bpt(\ca{L}),f^+(x^+)=0\}=\\
        & =\{\tilde{f}_{cf}\in \bpt(\Acf):\tilde{f}_{cf}(\del(\syp{x}))=1\}=\\
        & =\varphi^{-}_{\Acf}(\del(\syp{x})).\qedhere
    \end{align*}
\end{itemize}
\end{proof}

\begin{proposition}
For any finitary biframe $\ca{L}$ the bijection $\beta_{\ca{L}}$ is a bihomeomorphism $\mf{bpt}(\Acf)\cong Sk_{cf}(\mf{bpt}(\ca{L}))$.
\end{proposition}
\begin{proof}
To show that the bijection $\beta_{\ca{L}}$ is a bihomeomorphism, it suffices to show that the subbasic opens of each topology of $\bpt(\Acf)$ is the forward image of the subbasic opens of the corresponding topology of $Sk_{cf}(\bpt(\ca{L}))$. Each component of the bispatialization map turns finite meets into finite intersections and arbitrary joins into arbitrary unions. This means that a subbasis of the positive topology of $\bpt(\Acf)$ is given by 
\[
\{\varphi_{\Acf}^{+}(\na(\syp{x})):x^+\in L^+\}\cup \{\varphi_{\Acf}^{+}(\na(\sym{x})):x^-\in L^-\}.
\]
By definition of the closed-fitted Skula bispace, and by Lemma \ref{lemmabihomeo9cf}, these indeed are the forward images under $\beta_{\ca{L}}$ of the subbasic opens of the positive topology of $Sk_{cf}(\bpt(\ca{L}))$.
\end{proof}
 Finally, we check that the assignment $\beta:\mf{Obj}(\bd{BiFrm_{fin}})\ra \mf{Mor}(\bd{BiTop})$ defined as $\ca{L}\mapsto \beta_{\ca{L}}$ is a natural bijection.

\begin{theorem}
The following square commutes up to natural isomorphism.
\[
\begin{tikzcd}[row sep=large, column sep = large]
\bd{BiFrm_{fin}}^{op} 
\arrow{r}{\mf{bpt}} 
\arrow[swap]{d}{\mf{A}_{cf}} 
& \bd{BiTop} 
\arrow{d}{Sk_{cf}} \\
\bd{BiFrm_{fin}}^{op}   
\arrow{r}{\mf{bpt}} & 
\bd{BiTop}
\end{tikzcd}
\]
\end{theorem}
\begin{proof}
Suppose that there is a morphism $f:\ca{L}\ra \ca{M}$ in $\bd{BiFrm_{fin}}$. The naturality square is as follows.
\[
\begin{tikzcd}[row sep=large, column sep = large]
Sk_{cf}(\bpt(\ca{M}))
\arrow{r}{\beta_{\ca{M}}} 
\arrow[swap]{d}{Sk_{cf}(\bpt(f))} 
& \bpt(\mf{A}_{cf}(\ca{M}))
\arrow{d}{\bpt(\mf{A}_{cf}(f))} \\
Sk_{cf}(\bpt(\ca{L}))  
\arrow{r}{\beta_{\ca{L}}} & 
\bpt(\mf{A}_{cf}(\ca{L}))
\end{tikzcd}
\]
Since this amounts to an equality of functions, commutativity of this diagram amounts to commutativity of the following square in $\bd{Set}$. We have also used the definition of the functor $\bpt$.
\[
\begin{tikzcd}[row sep=large, column sep = large]
\bpt(\ca{M})
\arrow{r}{\beta_{\ca{M}}} 
\arrow[swap]{d}{-\circ f} 
& \bpt(\mf{A}_{cf}(\ca{M}))
\arrow{d}{-\circ\mf{A}(f)} \\
\bpt(\ca{L})) 
\arrow{r}{\beta_{\ca{L}}} & 
\bpt(\mf{A}_{cf}(\ca{L}))
\end{tikzcd}
\]
Suppose, then, that $g\in \bpt(\ca{M})$. We need to check that $\beta_{\ca{M}}(g)\circ \mf{A}_{cf}(f)=\beta_{\ca{L}}(g\circ f)$. To check that this is an equality of points of $\Acf$, it suffices to show that these two maps agree on where they map $\na(\syp{x})$, since every bipoint $h:\Acf\ra \bd{2}$ is completely determined to where it maps the closed congruences. For $x^+\in L^+$, we have the following chain of equalities.
\begin{align*}
    & \beta_{\ca{M}}(g)^+(\mf{A}_{cf}(f)^+(\na(\syp{x})))=\\
    & =\beta_{\ca{M}}(g)^+(\na(\sy{f^+(x^+)}))=\\
    & =\tilde{g}_{cf}^+(\na(\sy{f^+(x^+)}))=\\
    & =g^+(f^+(x^+))=\\
    & =(\widetilde{g\circ f})_{cf}^+(\na(\syp{x})).\qedhere
\end{align*}
\end{proof}

\subsection{The $\mf{A}_{\pm}$ functor and its spatial counterpart}

Recall that for a bispace $X$ its \tc{positive-negative Skula bispace} is the bispace $Sk_{\pm}(X)$ such that it has the same underlying set of points as $X$, whose positive topology is the one generated by the positive and opens of $X$ together with its negative closed sets (the monotopological Skula space of $(X,\Om^+(X))$), and whose negative topology is that generated by the negative opens of $X$ together with the negative closed ones (the monotopological Skula space of $(X,\Om^-(X))$). We now define the pointfree version of this functor, taking the duality of finitary biframes as our pointfree setting. For a finitary biframe $\ca{L}$, we define its \tc{positive-negative assembly} to be the biframe
\[
\Apm=(\sy{\mf{A}(L^+)},\sy{\mf{A}(L^-)},\mf{A}_{\fin}(L)).
\]
Here $\sy{\mf{A}(L^+)}$ is the image of $\mf{A}(L^+)$ under the map $\mf{A}(e^+_L):\mf{A}(L^+)\ra \mf{A}(L)$, where $e^+_L:L^+\ra L$ is the canonical embedding implicit in the definition of a biframe. We notice that all elements of $\sy{\mf{A}(L^+)}$ are finitary, and that since $\mf{A}(e^+_L)$ is a frame map, the subframe $\sy{\mf{A}(L^+)}$ of $\mf{A}_{\fin}(L)$ is the one generated by elements of the form $\na(\syp{a})$ together with those of the form $\del(\syp{a})$. 

\begin{proposition}\label{assemblyisopm}
For a finitary biframe $\ca{L}$ we have an isomorphism of biframes
\[
\Apm\cong (L^+\oplus \Fip,L^-\oplus \Fim,L^+\oplus \Fip,L^-\oplus \Fim)/(C_L\cup \mc{Com}^+_L\cup \mc{Com}^-_L),
\]
\end{proposition}
\begin{proof}
The fact that the two main components of our two biframes agree follows from Proposition \ref{assemblyiso}. The isomorphism given by this proposition acts as $\sy{\syp{x}}\mapsto \na(\syp{x})$ and $\sy{\syp{\up x}}\mapsto \del(\syp{x})$. This means that under this isomorphism the frame $L^+\oplus \Fip$ is mapped to the subframe of $\mf{A}_{\fin}(L)$ generated by the elements of the form $\na(\syp{x})$ together with those of the form $\del(\syp{x})$. This is the frame $\sy{\mf{A}(L^+)}$.
\end{proof}

\begin{proposition}
The closed-fitted assembly of a finitary biframe is finitary. 
\end{proposition}
\begin{proof}
This follows from Proposition \ref{assemblyisopm} and by Lemma \ref{finitarylemma9}.
\end{proof}

We then have the object part of the desired endofunctor. To define the assignment $\mf{A}_{\pm}$ on morphisms, we use the following lemma. For a biframe $\ca{L}$, we call \tc{positive} congruences the congruences on $L$ in the collection $\sy{\mf{A}(L^+)}$. For a biframe morphism $f:\ca{L}\ra \ca{M}$, we define the morphism $\mf{A}_{\pm}(f):\Apm\ra \mf{A}_{\pm}(\ca{M})$ as the pairing of the restriction of the frame map $\mf{A}(f):\mf{A}(L)\ra \mf{A}(M)$ to the two subframes $\sy{\mf{A}(L^+)}$ and $\sy{\mf{A}(L^-)}$. With the next lemma, we check that this is a well-defined biframe map.

\begin{lemma}
For a frame map $f:L\ra M$, the map $\mf{A}(f):\mf{A}(L)\ra \mf{A}(M)$ maps positive congruences to positive ones and negative ones to negative ones.
\end{lemma}
\begin{proof}
Since $\mf{A}(f)$ is a frame map, and by its definition, we have that it maps a positive congruence $\bve_i \na(\syp{x_i})\cap \del(\syp{y_i})$ to the congruence $\bve_i \na(\sy{f^+(x^+_i)})\cap \del(\sy{f^+(y^+_i)})$, which is a positive congruence of $M$.
\end{proof}

By its definition, the assignment $\mf{A}_{\pm}:\mf{Mor}(\bd{BiFrm_{fin}})\ra \mf{Mor}(\bd{BiFrm_{fin}})$ also respects identities and compositions. We have proved the following.

\begin{proposition}
The assignment $\mf{A}_{\pm}:\bd{BiFrm_{fin}}\ra \bd{BiFrm_{fin}}$ is a functor.
\end{proposition}

We now analyze the bispectrum of the positive-negative assembly of a finitary biframe. Because of Lemma \ref{assemblyiso} and \ref{assemblyisopm}, we know that $\Ad$ and $\Apm$ have the same patch. Similarly as in the closed-fitted assembly case, we observe that by the proof of Lemma \ref{twobiframessamepoints} we have a bijection $\bpt(\Ad)\cong \bpt(\Apm)$ given by $\tilde{f}\mapsto \tilde{f}_{\pm}$, where $\tilde{f}_{\pm}$ is the pairing of the restrictions of the frame map $\tilde{f}:\mf{A}_{\fin}(\ca{L})\ra 2$ to the two subframes $\sy{\mf{A}(L^+)}$ and $\sy{\mf{A}(L^-)}$. We have proved the following fact.

\begin{proposition}\label{bijectionassemblypm9}
The assignment $\gamma_{\ca{L}}:f\mapsto \tilde{f}_{\pm}$, with this map defined as
\begin{align*}
    & \tilde{f}_{\pm}:\Apm\ra \bd{2}\\
    & \na(\syp{x})\mapsto f^+(x^+)\\
    & \del(\syp{x})\mapsto \neg f^+(x^+),
\end{align*}
and similarly on the congruences of the form $\na(\sym{x})$ and $\del(\sym {x})$, constitutes a bijection $|\bpt(\ca{L})|\cong |\bpt(\mf{A}_{\pm}(\ca{L}))|$.
\end{proposition}

We now work towards showing that for every finitary biframe $\ca{L}$ the bijection $\gamma_{\ca{L}}$ actually is a bihomeomorphism $Sk_{\pm}(\bpt(\ca{L}))\cong \bpt(\Apm)$.

\begin{lemma}\label{lemmabihomeo9pm}
For any finitary biframe $\ca{L}$, we have the following.
\begin{itemize}
    \item $\gamma_{\ca{L}}[\va{L}{+}{x}]=\varphi_{\Apm}^+(\na(\syp{x}))$,
    \item $\gamma_{\ca{L}}[\va{L}{+}{x}^c]=\varphi_{\Apm}^+(\del(\syp{x}))$,
    \item $\gamma_{\ca{L}}[\va{L}{-}{x}]=\varphi_{\Apm}^-(\na(\sym{x}))$,
    \item $\gamma_{\ca{L}}[\va{L}{-}{x}^c]=\varphi_{\Apm}^-(\del(\sym{x}))$.
\end{itemize}
\end{lemma}
\begin{proof}
let us show the first and the second item.
\begin{itemize}
    \item For $x^+\in L^+$, we have the following chain of equalities. We have used the fact that $\gamma_{\ca{L}}:f\mapsto \tilde{f}_{\pm}$ is a bijection between the points of $\ca{L}$ and those of $\Apm$, and so it is surjective. 
    \begin{align*}
        & \gamma_{\ca{L}}[\va{L}{+}{x}]=\\
        & =\{\tilde{f}_{\pm}\in \bpt(\Apm): f\in \bpt(\ca{L}),f^+(x^+)=1\}=\\
        & =\{\tilde{f}_{\pm}\in \bpt(\Apm):\tilde{f}_{\pm}(\na(\syp{x}))=1\}=\\
        & =\varphi^{+}_{\Apm}(\na(\syp{x})).
    \end{align*}
    \item For $x^+\in L^+$, we have the following chain of equalities. For the equality between the second and the third line, we have used the fact that by Lemma \ref{bijectionpointscf} we know that $\gamma_{\ca{L}}:f\mapsto \tilde{f}_{cf}$ is a bijection between the points of $\ca{L}$ and those of $\Apm$, and so it is surjective. 
    \begin{align*}
        & \gamma_{\ca{L}}[\va{L}{+}{x}^c]=\\
        & =\{\tilde{f}_{\pm}\in \bpt(\Apm): f\in \bpt(\ca{L}),f^+(x^+)=0\}=\\
        & =\{\tilde{f}_{\pm}\in \bpt(\Apm):\tilde{f}_{\pm}(\del(\syp{x}))=1\}=\\
        & =\varphi^{+}_{\Apm}(\del(\syp{x})).\qedhere
    \end{align*}
\end{itemize}
\end{proof}

\begin{proposition}
For any finitary biframe $\ca{L}$ the bijection $\beta_{\ca{L}}$ is a bihomeomorphism $Sk_{\pm}(\mf{bpt}(\ca{L}))\cong \mf{bpt}(\Apm)$.
\end{proposition}
\begin{proof}
To show that the bijection $\gamma_{\ca{L}}$ is a bihomeomorphism, it suffices to show that the subbasic opens of each topology of $\bpt(\Apm)$ is the forward image of the subbasic opens of the corresponding topology of $Sk_{\pm}(\bpt(\ca{L}))$. Each component of the bispatialization map turns finite meets into finite intersections and arbitrary joins into arbitrary unions. This means that a subbasis of the positive topology of $\bpt(\Apm)$ is given by 
\[
\{\varphi_{\Apm}^{+}(\na(\syp{x})):x^+\in L^+\}\cup \{\varphi_{\Apm}^{+}(\del(\syp{x})):x^-\in L^-\}.
\]
By definition of the positive-negative Skula bispace, and by Lemma \ref{lemmabihomeo9pm}, these indeed are the forward images under $\gamma_{\ca{L}}$ of the subbasic opens of the positive topology of $Sk_{\pm}(\bpt(\ca{L}))$.
\end{proof}
 Finally, we check that the assignment $\gamma:\mf{Obj}(\bd{BiFrm_{fin}})\ra \mf{Mor}(\bd{BiTop})$ defined as $\ca{L}\mapsto \gamma_{\ca{L}}$ is a natural isomorphism.

\begin{theorem}
The following square commutes up to natural isomorphism.
\[
\begin{tikzcd}[row sep=large, column sep = large]
\bd{BiFrm_{fin}}^{op} 
\arrow{r}{\mf{bpt}} 
\arrow[swap]{d}{\mf{A}_{\pm}} 
& \bd{BiTop} 
\arrow{d}{Sk_{\pm}} \\
\bd{BiFrm_{fin}}^{op}   
\arrow{r}{\mf{bpt}} & 
\bd{BiTop}
\end{tikzcd}
\]
\end{theorem}
\begin{proof}
Suppose that there is a morphism $f:\ca{L}\ra \ca{M}$ in $\bd{BiFrm_{fin}}$. The naturality square amounts to commutativity of the following square in $\bd{Set}$. We have also used the definition of the functor $\bpt$.
\[
\begin{tikzcd}[row sep=large, column sep = large]
\bpt(\ca{M})
\arrow{r}{\gamma_{\ca{M}}} 
\arrow[swap]{d}{-\circ f} 
& \bpt(\mf{A}_{\pm}(\ca{M}))
\arrow{d}{-\circ\mf{A}_{\pm}(f)} \\
\bpt(\ca{L})) 
\arrow{r}{\gamma_{\ca{L}}} & 
\bpt(\mf{A}_{\pm}(\ca{L}))
\end{tikzcd}
\]
Suppose, then, that $g\in \bpt(\ca{M})$. We need to check that $\gamma_{\ca{M}}(g)\circ \mf{A}_{\pm}(f)=\gamma_{\ca{L}}(g\circ f)$. We need to show that these two expressions denote the same point of $\Apm$; to do this we show that these two maps agree on where they map $\na(\syp{x})$, since every of $\Apm$ is completely determined by where it maps the closed congruences. For $x^+\in L^+$, we have the following chain of equalities.
\begin{align*}
    & \gamma_{\ca{M}}(g)^+(\mf{A}_{\pm}(f)^+(\na(\syp{x})))=\\
    & =\gamma_{\ca{M}}(g)^+(\na(\sy{f^+(x^+)}))=\\
    & =\tilde{g}_{\pm}^+(\na(\sy{f^+(x^+)}))=\\
    & =g^+(f^+(x^+))=\\
    & =(\widetilde{g\circ f})_{\pm}^+(\na(\syp{x})).\qedhere
\end{align*}
\end{proof}

\section{Finitary fitness and subfitness}

In \cite{picado14}, biframe versions of fitness and subfitness are introduced. We introduce two other biframe notions of fitness and subfitness which work well as analogues of these for the finitary duality. We say that a finitary biframe $\ca{L}$ is \tc{subfit} if whenever $a\we \syw{x}{x}\nleq b$ there are $y^+\in L^+$ and $y^-\in L^-$ such that the following holds
\[
\begin{cases} 
(\syw{x}{x})\ve \syv{y}{y}=1,\\ 
a\nleq \syv{y}{y}\ve b.
\end{cases}
\]

In \cite{cleme18}, the \tc{fitting} operation on the coframe $\mf{S}(L)$ of sublocales of a frame is introduced. In the language of congruence, this operation $\mathit{fit}:\mf{A}(L)\ra \mf{A}(L)$ sends a congruence $C$ to its ``fitted interior" $\bve \{\del(x):\del(x)\se C\}$. For a finitary biframe $\ca{L}$, we define the \tc{finitary fitting} of a finitary congruence $C$ on $L$ as the congruence 
\[
\bve \{\del(\syv{x}{x}):\del(\syv{x}{x})\se C\}.
\]
It is clear that this is an interior operator on the frame $\mf{A}_{\fin}(L)$ of finitary congruences on $L$. In \cite{picadopultr2011frames} the following characterization theorem is given.
\begin{theorem}
For a frame $L$, the following are equivalent.
\begin{enumerate}
    \item $L$ is subfit.
    \item For every $x\in L$ we have that $\del(x)=\bca \{\na(y):x\ve y=1\}$.
    \item Every closed congruence on $L$ is fitted.
    \item Any congruence on $L$ with trivial fitting is trivial. 
\end{enumerate}
\end{theorem}
We will shortly prove an analogous theorem for finitary biframes. 
\begin{lemma}\label{congruenceinclusion}
For a frame $L$ and for $x,y\in L$, we have that $\del(x)\se \na(y)$ if and only if $x\ve y=1$.
\end{lemma}
\begin{proof}
It suffices to notice that both conditions are equivalent to $\del(x)\cap \del(y)=L\times L$.
\end{proof}
We notice that the analogue of characterization (4) does not hold for the biframe versions of fitness and subfitness in \cite{picado14}, but we are able to recover it if we move to the finitary setting.

\begin{theorem}\label{subfitness}
For a finitary biframe $\ca{L}$, the following are equivalent.
\begin{enumerate}
    \item The biframe $\ca{L}$ is subfit.
    \item For every $x^+\in L^+$ and every $x^-\in L^-$ we have that $\del(\syw{x}{x})=\bca\{\na(\syv{y}{y}):(\syw{x}{x})\ve \syv{y}{y}=1\}$.
    \item Every open congruence on $L$ is an intersection of finitary closed congruences.
    \item Any finitary congruence on $L$ with trivial finitary fitting is trivial.
\end{enumerate}
\end{theorem}
\begin{proof}
Let $\ca{L}$ be a finitary biframe.

(1)$\implies$(2). Suppose that $\ca{L}$ is subfit. Let $x^+\in L^+$ and let $x^-\in L^-$. If $(\syw{x}{x})\ve \syv{y}{y}=1$ then we have $\del(\syw{x}{x})\se \na(\syv{y}{y})$ by Lemma \ref{congruenceinclusion}, and so we have one of the desired set inclusions. For the other, suppose that we have $(a,b)\notin \del(\syw{x}{x})$. This means that we have $a\we \syw{x}{x}\neq b\we \syw{x}{x}$, and in particular we have that $a\we \syw{x}{x}\nleq b$. Then, by subfitness, there are $y^+\in L^+$ and $y^-\in L^-$ such that both $(\syw{x}{x})\ve \syv{y}{y}=1$ and such that $a\nleq \syv{y}{y}\ve b$. This means that $(a,b)\notin \na(\syv{y}{y})$, and so indeed we have the other inclusion.\\[2mm]
(2)$\implies$(1). Suppose that (1) holds, and that we have $a\we \syw{x}{x}\nleq b$. This means in particular that $a\we \syw{x}{x}\nleq b\we \syw{x}{x}$, and by definition of open congruences this means that $\na(a)\cap \del(b)\nsubseteq \del(\syw{x}{x})$. By our assumption, this implies that there are $y^+\in L^+$ and $y^-\in L^-$ such that $(\syw{x}{x})\ve \syv{y}{y}=1$ and such that $\na(a)\cap \del(b)\nsubseteq \na(\syv{y}{y})$. This latter condition means that $\na(a)\nsubseteq \na(\syv{y}{y})\ve \na(b)$, and this is equivalent to $a\neq \syv{y}{y}\ve b$. Thus, the finitary biframe is subfit. \\[2mm]
(2)$\implies$ (3). This follows directly from the fact that every open finitary congruence is a finite intersection of congruences of the form $\del(\syw{x}{x})$.  \\[2mm]
(3)$\implies$ (4). Suppose that every open congruence on $L$ is an intersection of closed congruences. Now, suppose that we have a finitary congruence $C$ such that its fitting is $L$, that is, suppose that we have $C$ such that $\del(\syv{x}{x})\se C$ implies that $\syv{x}{x}=1$ for every $x^+\in L^+$ and every $x^-\in L^-$. As $C$ is finitary, to show that it is the identity on $L$ it suffices to show that $\na(\syw{x}{x})\cap \del(\syv{y}{y})\se C$ implies that $\syw{x}{x}\leq \syv{y}{y}$ in $L$. Suppose, then, that the antecedent of this proposition holds. We have that $\del(\syv{y}{y})\se C\ve \del(\syw{x}{x})$. By assumption, $\del(\syw{x}{x})=\bca \{\na(\syv{y_i}{y_i}):i\in I\}$ for some collection $y^+_i\in L^+$ and $y^-_i\in L^-$. We substitute into our assumption and we get $\del(\syv{y}{y})\se C\ve \bca_i \na (\syv{y_i}{y_i})\se \bca (C\ve \na(\syv{y_i}{y_i}))$. This means that for every $i\in I$ we have that $\del(\syv{y}{y})\cap \del(\syv{y_i}{y_i})=\del(\syv{y}{y}\ve \syv{y_i}{y_i})\se C$. By assumption on the congruence $C$ this means that $\syv{y}{y}\ve \syv{y_i}{y_i}=1$. We obtain, by Lemma \ref{congruenceinclusion}, that $\del(\syv{y}{y})\se \bca_i \na(\syv{y_i}{y_i})=\del(\syw{x}{x})$, from which we deduce $\syw{x}{x}\leq \syv{y}{y}$. Then, indeed, $C$ must be the identity $L\times L$. \\[2mm]
(4)$\implies$(2). Suppose that every finitary congruence with trivial fitting is trivial Let $x^+\in L^+$ and let $x^-\in L^-$. Let $\{\syv{y_i}{y_i}:i\in I\}$ be the collection $\{\syv{y}{y}:(\syw{x}{x})\ve \syv{y}{y}=1\}$. We show that $\bca_i\na(\syv{y_i}{y_i})\se \del(\syw{x}{x})$ by showing that we have $\bca_i\na(\syv{y_i}{y_i})\cap \na(\syw{x}{x})\se L\times L$. In virtue of our assumption, it suffices to show that this congruence has a trivial fitting. Suppose, then, that there are $a^+\in L^+$ and $a^-\in L^-$ such that 
\begin{align*}
    \del(\syv{a}{a})\se \bca_i \na(\syv{y_i}{y_i})\cap \na(\syw{x}{x})& & (*)
\end{align*}
In particular, $\del(\syv{a}{a})\se \na(\syw{x}{x})$, which by Lemma \ref{congruenceinclusion} means that $(\syw{x}{x})\ve \syv{a}{a}=1$, and so $\syv{a}{a}$ must be in the collection $\{\syv{y_i}{y_i}:i\in I\}$. Say $\syv{a}{a}=\syv{y_j}{y_j}$. But by $(*)$ we must also have that $\del(\syv{a}{a})=\del(\syv{y_j}{y_j})\se \na(\syv{y_j}{y_j})$. We must then have that $\del(\syv{y_j}{y_j})\cap \del(\syv{y_j}{y_j})=L\times L$, and so we must have that $\syv{a}{a}=1$. \qedhere
\end{proof}

Let us move to the analogue of fitness for finitary biframes. We say that a finitary biframe $\ca{L}$ is \tc{fit} if whenever $\syw{x}{x}\nleq a$ there are $y^+\in L^+$ and $y^-\in L^-$ such that the following holds
\[
\begin{cases} 
(\syw{x}{x})\ve \syv{y}{y}=1,\\ 
(\syv{y}{y})\ra a\neq a.
\end{cases}
\]
We will give right away a characterization of biframe fitness in terms of sublocales, which perhaps is more transparent than the frame theoretical definition, especially in light of the analogous characterization of subfitness.

In \cite{picadopultr2011frames}, the following is shown.
\begin{theorem}
For a frame $L$, the following are equivalent.
\begin{enumerate}
    \item The frame $L$ is fit.
    \item For every $x\in L$ we have that $\na(x)=\bve \{\del(y):x\ve y=1\}$.
    \item If two congruence on $L$ have the same fitting, they are the same congruence.
    \item Every closed congruence of $L$ is fitted.
    \item Every congruence on $L$ is fitted.
    \item All quotients of $L$ are fit.
\end{enumerate}
\end{theorem}
Once again, we prove a biframe analogue of this theorem. The biframe analogue of characterization (3) does not hold for fitness as defined in \cite{picado14}, but we do recover it if we stay in the finitary setting.
\begin{remark}
In the next proof, we will temporarily work with sublocales rather than with congruences. We recall that the lattice of sublocales of a frame $L$ is anti-isomorphic to the frame of congruences on it. The anti-isomorphism maps a congruence $C$ to the subset $\{\bve [x]_C:x\in L\}\se L$ of all top elements of the equivalence classes of $C$. In particular, we have that an open congruence $\del(a)$ is mapped by this anti-isomorphism to the set $\op (a)=\{a\ra x:x\in L\}$, and we also have that $b\in \op (a)$ if and only if $a\ra b=b$. A closed congruence $\na(a)$ is mapped to the sublocale $\up a\se L$, and so we have that $b\in \cl(a)$ if and only if $a\leq b$. In the next proof, we will have to show an inclusion of the form $\na(x)\se \bve_i \del(y_i)$. In order to do this, we will work within the lattice of sublocales, and we will show that $\bca_i \op (y_i)\se \cl(x)$.
\end{remark}
\begin{theorem}\label{fitness}
For a finitary biframe $\ca{L}$, the following are equivalent.
\begin{enumerate}
    \item The biframe $\ca{L}$ is fit.
    \item For every $x^+\in L^+$ and every $x^-\in L^-$ we have that $\na(\syw{x}{x})=\bve \{\del(\syv{y}{y}):(\syw{x}{x})\ve \syv{y}{y}=1\}$.    
    \item If two finitary congruences on $L$ have the same finitary fitting, the two congruences are equal.
    \item Every closed congruence of $L$ is finitary fitted.
    \item Every finitary congruence on $L$ is finitary fitted.
    \item All biquotients of $\ca{L}$ are fit.
\end{enumerate}
\end{theorem}
\begin{proof}
(1) is equivalent to (2). Both conditions are equivalent to having that $a\notin \cl(\syw{x}{x})$ implies that $a\notin \bca \{\op (\syv{y}{y}):(\syw{x}{x})\ve\syv{y}{y}=1\}$.\\[2mm]
(2)$\implies$ (3). Suppose that $C$ and $D$ are two finitary congruences on $L$ such that for every $a^+\in L^+$ and every $a^-\in L^-$ we have that $\del(\syv{a}{a})\se C$ if and only if $\del(\syv{a}{a})\se D$. Suppose that we have $\na(\syw{a}{a})\cap \del(\syv{b}{b})\se C$. We show that this implies $\na(\syw{a}{a})\cap \del(\syv{b}{b})\subseteq D$. By our assumption, we have that $\na(\syw{a})=\bve \del(\syv{a_i}{a_i})$, where $\{\syv{a_i}{a_i}:i\in I\}=\{\syv{x}{x}:(\syw{a}{a})\ve\syv{x}{x}=1\}$. We then have that $\bve_i \del(\syv{a_i}{a_i})\cap \del(\syv{b}{b})\se C$, which means that for every $i\in I$ we have $\del(\syv{a_i}{a_i}\ve \syv{b}{b})\se C$, and by our assumption this implies that also $\del(\syv{a_i}{a_i}\ve \syv{b}{b})\se D$. But this implies, in turn, that $\bve _i \del(\syv{a_i}{a_i})\cap \del(\syv{b}{b})=\na(\syw{a}{a})\cap \del(\syv{b}{b})\se D$.\\[2mm]
(3)$\implies$ (2). Suppose that finitary congruences with the same finitary fitting are the same congruence. By Lemma \ref{congruenceinclusion}, it suffices to show $\na(\syw{x}{x})\se \bve_i \del(\syv{y_i}{y_i})$, in which $\{\syv{y_i}{y_i}:i\in I\}=\{\syv{y}{y}:(\syw{x}{x})\ve \syv{y}{y}=1\}$. Because of our assumption, to show this fact it suffices to show that $\del(\syv{a}{a})\se \na(\syw{x}{x})$ implies that $\del(\syv{a}{a})\se \bve _i \del(\syv{y_i}{y_i})$. Suppose, then, that the antecedent holds. By Lemma \ref{congruenceinclusion}, we then have that $(\syv{x}{x})\ve\syv{a}{a}=1$, and definition of this collection, we have that $\syw{a}{a}$ is in $\{\syv{y_i}{y_i}:i\in I\}$. Then, indeed the consequent holds too. \\[2mm]
(2)$\implies$(4). This clearly, holds, since all closed congruences are joins of congruences of the form $\del(\syw{x}{x})$.\\[2mm]
(4) and (5) are equivalent. Only one direction is nontrivial. This holds because every finitary congruence is of the form $\bve_i \na(\syw{x_i}{x_i})\cap \del(\syv{y_i}{y_i})$. \\[2mm]
(5)$\implies$(6). Suppose that all finitary congruences on $\ca{L}$ are finitary fitted. It suffices to show that all biquotients of $\ca{L}$ are such that the finitary congruences on them are all finitary fitted. Suppose, then, that $C$ is a finitary congruence on $L$. By assumption, it is $\bve_i\del(\syv{x_i}{x_i})$ for collections $x^+_i\in L^+$ and $x^-_i\in L^-$. Suppose, now, that $D$ is a finitary congruence on $L/C$. Now, for every equivalence class of the form $[\syp{x}]_C$ we select the witness $w(\syp{x}):=\bve \{\syp{y}:\syp{y}\in [\syp{x}]_C\}$, and we select similarly canonical witnesses from the negative equivalence classes. By Lemma \ref{witness}, we have that the quotient $(L/C)/D$ is the same as $L$ quotiented by $C\cup \{(w(\syp{a})\we w(\sym{a}),w(\syp{b})\ve w(\sym{b}):(\syw{a}{a},\syv{b}{b})\in D\}$. This is a finitary congruence on $L$, and so by assumption we have that it is of the form $\bve_j \del(\syv{y_j}{y_j})$. Then, $(L/C)/D$ is $L$ quotiented by
\[
\bve_i \del(\syv{x_i}{x_i})\cap \bve_j\del(\syv{y_j}{y_j})
\]
Therefore, by Lemma \ref{finitarylemma91}, the quotient $(L/C)/D$ is the same as quotienting $L/C$ by the congruence $\bve \{\del([\syp{y_j}]_C\ve [\sym{y_j}]_C):j\in J\}$. Indeed, then, $D$ is a finitary fitted congruence on $L/C$.
\\[2mm]
(6)$\implies$(1). This is clear. \qedhere
\end{proof}

Fitness and subfitness were both introduced as possible pointfree versions of the $T_1$ axiom. It is then  natural to also explore the characterizations of those finitary biframes such that their bispectrum is pairwise $T_1$. We follow \cite{salbany74} in defining a bispace to be \tc{pairwise} $T_1$, or simply $\tc{bi}-T_1$, if whenever we have $x,y\in X$ with $x\neq y$ there is some $U\in \Om^+(X)\cup \Om^-(X)$ such that it contains $x$ and omits $y$. 
\begin{lemma}\label{pairwiset1}
A bispace $C$ is pairwise $T_1$ if and only if for every $x\in X$ we have that $\{x\}=U^+{}^c\cap U^-{}^c$ for some positive open $U^+$ and some negative open $U^-$.
\end{lemma}
\begin{proof}
Suppose that $X$ is pairwise $T_1$, and let $x\in X$. For every $y\in Y$ with $x\neq y$ there is either some $U^+_y$ or some $U^-_y$ such that it contains $y$ and omits $x$. We may then consider the set $(\bcu \{U^+{}_y\cup U^-{}_y:y\neq x\})^c$. This is a set of the form $U^+{}^c\cap U^-{}^c$ which only contains $x$. For the converse, suppose that every singleton is of the form $U^+{}^c\cap U^-{}^c$ and that $x\neq y$. Then, we have that $y\in U^+\cup U^-$ and so either $y\in U^+$ or $y\in U^-$. We also have that $x\notin U^+$ and $x\notin U^-$, and so we are done.
\end{proof}

Observe how the following theorem may be read as the spatial version of Theorem \ref{fitness}.

\begin{theorem}\label{pairwiseT1}
The following are equivalent for a d-frame $\ca{L}$.
\begin{enumerate} 
\item The bispace $\bpt(\ca{L})$ is pairwise $T_1$. 
    \item For each $f\in \bpt(\ca{L})$ we have $\{f\}=\varphi_{\ca{L}}^+(a^+){}^c\cap \varphi_{\ca{L}}^-(a^-){}^c$.
     \item The negative topology of $Sk_{cf}(\bpt(\ca{L}))$ is discrete.
    \item All bisubspaces of $\bpt(\ca{L})$ are bispectra of finitary fitted biquotients of $\ca{L}$.
    \end{enumerate}
\end{theorem}
\begin{proof}
(1) and (2) are equivalent. This follows from Lemma \ref{pairwiseT1}.\\[2mm]
(2) and (3) are equivalent. This follows from the fact that the sets of the form $\varphi_{\ca{L}}^+(a^+){}^c\cap \varphi_{\ca{L}}^-(a^-){}^c$ are basic open sets of the negative topology of $Sk_{cf}(\bpt(\ca{L}))$.\\[2mm]
(2)$\implies$(4). If (2) holds, then an arbitrary subspace of $\bpt(\ca{L})$ is of the form $(\bcu_i \varphi_{\ca{L}}^+(a^+_i){}^c\cap \varphi_{\ca{L}}^-(a^-_i){}^c)^c=\bca_i \va{L}{+}{a_i}\cup \va{L}{-}{a_i}$. By items (1) and (4) of Proposition \ref{manyfacts9}, this is the underlying set of the bispectrum of $\ca{L}/\bve_i \del(\syv{a_i}{a_i})$.\\[2mm]
(4) $\implies$ (2). If all subspaces are bispectra of finitary fitted biquotients, then in particular for every $f\in \bpt(\ca{L})$ the subspace $\{f\}^c$ if the bispectrum of $\ca{L}/\bve_i \del(\syv{a_i}{a_i})$ for some families $a^+_i\in L^+$ and $a^-_i\in L^-$. By items (1) and (4) of Proposition \ref{manyfacts9}, this is the same as $\bca_i \va{L}{+}{a_i}\cup \va{L}{-}{a_i}$. Hence, we have that $\{f\}=\bcu_i \va{L}{+}{a_i}^c\cap \va{L}{-}{a_i}^c$, which means that there must be some $j\in I$ with $\{f\}=\va{L}{+}{a_j}^c\cap \va{L}{-}{a_j}^c$.
\end{proof}

\bibliographystyle{elsarticle-harv}
\bibliography{bibliovariation}
\end{document}